\documentclass[final]{colt2025} 
\usepackage{algorithm, algorithmic}
\usepackage[T1]{fontenc}
\usepackage{mathtools}
\mathtoolsset{showonlyrefs}

\usepackage{xcolor}
\usepackage{pgfplots}
\pgfplotsset{compat=1.18}
\usepgfplotslibrary{external}
\PassOptionsToPackage{compress,elide}{natbib}

\newtheorem{assumption}{Assumption}

\usepackage{mathtools}
\usepackage{graphicx}
\usepackage{etoolbox}
\usepackage{amsmath, amssymb}

\newcommand{\dR}{\mathbb{R}}

\newcommand{\dC}{\mathbb{C}}

\DeclarePairedDelimiterXPP{\Pb}[1]{\mathbb{P}}{\lparen}{\rparen}{}{ #1}
\DeclarePairedDelimiterXPP{\E}[1]{\mathbb{E}}[]{}{ #1}
\DeclarePairedDelimiterX{\Set}[1]\lbrace\rbrace{ #1}

\DeclareMathOperator{\rank}{rank}

\DeclarePairedDelimiterX{\norm}[1]\lVert\rVert{\ifblank{#1}{\: \cdot \:}{#1}}

\newcommand{\bilingualcommand}[3]{%
  \newcommand{#1}[1][\ ]{%
    ##1%
    \iflanguage{english}{\text{#2}}{%
      \iflanguage{french}{\text{#3}}{}%
    }%
    ##1%
  }%
}

\bilingualcommand{\where}{where}{où}
\bilingualcommand{\textif}{if}{si}
\bilingualcommand{\textand}{and}{et}
\bilingualcommand{\textiff}{if and only if}{si et seulement si}
\bilingualcommand{\otherwise}{otherwise}{sinon}

\newcommand{\eps}{\varepsilon}

\newcommand{\quand}{\quad \textand{} \quad}

\usepackage[figure]{hypcap}

\makeatletter
\patchcmd{\@thm}{\@kernel@refstepcounter}{\refstepcounter}{}{}
\makeatother

\title{Community detection with the Bethe-Hessian}
\usepackage{times}

\coltauthor{%
  \Name{Ludovic Stephan} \Email{ludovic.stephan@ensai.fr}\\
  \addr Univ Rennes, Ensai, CNRS, CREST – UMR 9194
  F-35000 Rennes, France
  \AND
  \Name{Yizhe Zhu} \Email{yizhezhu@usc.edu}\\
  \addr Department of Mathematics
  University of Southern California%
}

\begin{document}

\maketitle

\begin{abstract}%
  The Bethe-Hessian matrix, introduced by Saade, Krzakala, and Zdeborová \citep{saade2014spectral}, is a Hermitian matrix designed for applying spectral clustering algorithms to sparse networks. Rather than employing a non-symmetric and high-dimensional non-backtracking operator, a spectral method based on the Bethe-Hessian matrix is conjectured to also reach the Kesten-Stigum detection threshold in the sparse stochastic block model (SBM). We provide the first rigorous analysis of the Bethe-Hessian spectral method in the SBM under both the bounded expected degree and the growing degree regimes. Specifically, we demonstrate that: (i) When the expected degree $d\geq 2$, the number of negative outliers of the Bethe-Hessian matrix can consistently estimate the number of blocks above the Kesten-Stigum threshold, thus confirming a conjecture from \citep{saade2014spectral} for $d\geq 2$. (ii) For sufficiently large $d$, its eigenvectors can be used to achieve weak recovery. (iii) As $d\to\infty$, we establish the concentration of the locations of its negative outlier eigenvalues, and weak consistency can be achieved via a spectral method based on the Bethe-Hessian matrix. \footnote{Accepted for presentation at the Conference on Learning Theory (COLT) 2025.}
\end{abstract}


\section{Introduction}\label{sec:intro}

The challenge of recovering community structures in networks has spurred significant advancements in spectral algorithms for very sparse graphs. A popular model for studying community detection on random graphs is the stochastic block model (SBM), first introduced in \citep{holland1983stochastic}. The SBM is a generative model for random graphs with a community structure, serving as a useful benchmark for clustering algorithms on graph data. When the random graph generated by an SBM is sparse with bounded expected degrees, a phase transition has been observed around the so-called \textit{Kesten-Stigum threshold}.  In particular, above this threshold, a wealth of algorithms are known to achieve weak recovery (better than a random guess) \citep{mossel2018proof,abbe2018proof,coja2018information,hopkins2017efficient,ding2022robust}. Several spectral algorithms have been proposed based on different matrices associated with the SBM, including self-avoiding \citep{massoulie2014community}, non-backtracking \citep{decelle2011asymptotic, krzakala2013spectral, bordenave2018nonbacktracking}, graph powering \citep{abbe2020graph}, or distance \citep{stephan2019robustness} matrices. For additional references and a more in-depth discussion of the Kesten-Stigum detection threshold, we refer interested readers to the survey of \citep{abbe2018community}.

For spectral clustering in random graphs with constant expected degree, a popular choice is the \textit{non-backtracking operator} \citep{krzakala2013spectral,bordenave2018nonbacktracking}. In the case of the SBM, it has been shown \citep{bordenave2018nonbacktracking} that the spectral method based on the non-backtracking operator can achieve the Kesten-Stigum threshold. In recent years, the non-backtracking operator has been crucial for analyzing the spectrum of sparse random matrices \citep{benaych2020spectral,alt2021extremal,dumitriu2022extreme,bordenave2020new,bordenave2019eigenvalues} and for enabling effective low-rank matrix recovery under sparse noise \citep{bordenave2020new,stephan2022non,stephan2024non}. Compared to the more classical adjacency matrix or Laplacian matrix, the non-backtracking operator is non-Hermitian, and its spectrum is more informative and stable when the underlying graph is very sparse.

Recent works on the spectra of sparse random matrices \citep{benaych2019largest,alt2024localized,hiesmayr2023spectral} show that, when the expected degree is constant, the top eigenvalues and eigenvectors of the adjacency matrix are not informative about the partition. This phenomenon is known as \textit{eigenvector localization} in random matrix theory for sparse Hermitian random matrices. Removing high-degree vertices can mitigate localization \citep{feige2005spectral,chin2015stochastic,le2017concentration}, but it causes significant information loss in the original graph, and approaches based on regularized adjacency or Laplacian matrices have yet to reach the fundamental detection threshold \citep{le2017concentration}.

On the other hand, a general phenomenon is that the spectrum of non-Hermitian random matrices is less sensitive to rows or columns with large $\ell_2$-norms \citep{benaych2020spectral,coste2023sparse,bordenave2022convergence,he2023edge} compared to the Hermitian ones. This partially explains why the non-backtracking operator performs better on very sparse graphs compared to other Hermitian operators. This \textit{``asymmetry helps"} principle has been leveraged to design novel algorithms with improved performance \citep{chen2021asymmetry,bordenave2020detection,stephan2024non} in low-rank matrix recovery problems.

The non-backtracking matrix is of size $2m\times 2m$, where $m$ is the number of edges in a graph, rather than $n\times n$, where $n$ is the number of vertices. When the average degree of a graph is large, finding the eigenvalues and eigenvectors of such a high-dimensional matrix is computationally expensive. Although, as shown in \citep{krzakala2013spectral,stephan2022sparse}, we can work with a smaller $2n \times 2n$ non-Hermitian block matrix using the Ihara-Bass formula \citep{bass_iharaselberg_1992}, this approach is still less efficient than the spectral method on $n \times n$ Hermitian matrices. Additionally, linear algebra methods are generally faster and more stable for symmetric matrices than for non-symmetric ones. More importantly, the reduction to a $2n \times 2n$ matrix does not apply to weighted graphs \citep{bordenave2020detection,stephan2022non}, as the Ihara-Bass formula becomes more complex and requires an additional parameter \citep{benaych2020spectral}.

Several $n\times n$ Hermitian matrices, such as the self-avoiding matrix \citep{massoulie2014community}, the distance matrix \citep{stephan2019robustness}, and the graph powering matrix \citep{abbe2020graph}, have been shown to reach the Kesten–Stigum threshold with spectral methods. However, each of these approaches requires setting a hyper-parameter $\ell = c \log n$ to transform the adjacency matrix $A$ into a new matrix $A^{(\ell)}$, which is essentially a modified version of \(A^\ell\). The preprocessing step needed to compute $A^{(\ell)}$ has a time complexity of \(O(n^{1+\kappa})\) for some constant \(\kappa > 0\), rendering these methods impractical for large-scale problems.

This raises a natural question: is there a solution that offers the best of both worlds?
\begin{center} \textit{Is there an efficient spectral method with an $n \times n$ Hermitian matrix that performs as well as the non-backtracking matrix for community detection?}
\end{center}

\paragraph{The Bethe-Hessian matrix}
Such an operator exists in the statistical physics literature, known as the \textit{Bethe-Hessian matrix}, proposed by Saade, Krzakala, and Zdeborová in \citep{saade2014spectral}. This operator is a linear combination of the adjacency matrix
$A$ and the diagonal degree matrix
$D$, with one parameter $t\in \dR$ given by:
\begin{equation}\label{eq:def_H}
  H(t):=t^2I-t A+(D-I).
\end{equation}
$H(t)$ is also called a deformed graph Laplacian \citep{grindrod2018deformed}.
In the context of SBM, the parameter $t$ is chosen as $t=\pm\sqrt{d}$ in \citep{saade2014spectral}, where $d$ is the average expected degree, and it can be estimated by the empirical mean degree. It was conjectured in \citep{saade2014spectral} that the number of negative outlier eigenvalues of $H(\pm \sqrt d)$ can consistently estimate the number of assortative and disassortative communities respectively above the Kesten-Stigum threshold and the eigenvectors associated with these negative outliers can be used to detect the communities.

\begin{figure}
  \centering
  \includegraphics[width=0.7\textwidth]{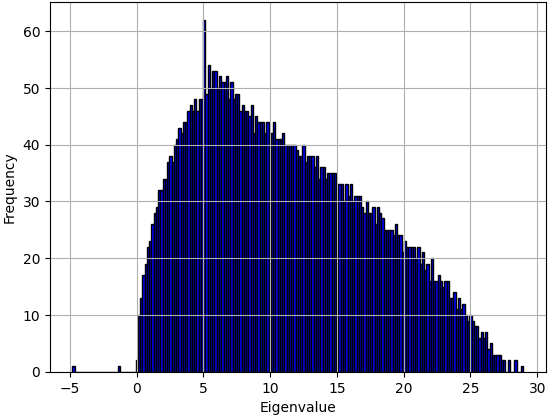}
  \caption{Eigenvalue distributions of the Bethe-Hessian matrix $H(\sqrt{d})$ for a 2-block SBM, where $n=4000,P_{11}=P_{22}=10$, and $P_{12}=P_{21}=2$. There are two negative outliers.}
  \label{fig:NB_BH}
\end{figure}

The Bethe-Hessian method is perhaps the simplest algorithm conjectured to achieve the Kesten-Stigum threshold in the SBM. Since the pioneering work in \citep{saade2014spectral}, it has been widely used in spectral clustering \citep{dall2019revisiting,dall2021nishimori,dall2021unified}, change point detection \citep{hwang_bethe_2022}, and estimating the number of communities in networks \citep{le2022estimating,hwang2023estimation,shao2024determining}.  The popular network analysis package NetworkX \citep{SciPyProceedings_11} includes this algorithm.

From a random matrix theory perspective, the Bethe-Hessian matrix introduces a new phenomenon not explained by existing theory. Results in sparse random matrices  \citep{alt2021delocalization,alt2021extremal,alt2023poisson,ducatez2023spectral} suggest that the Hermitian matrix $H(\pm\sqrt d)$ should have many outlier eigenvalues and suffer from eigenvector localization. However,  as shown in Figure~\ref{fig:NB_BH}, there is  \textit{separation between informative and uninformative outliers}.  The localization effect does appear in $H(\pm\sqrt d)$: there are many large positive eigenvalues, and their corresponding eigenvectors are localized on high-degree vertices. The striking phenomenon on the negative real line, however, still lacks a theoretical explanation:
\begin{center}
  \textit{Why do the negative outlier eigenvalues and eigenvectors in $H(\pm \sqrt d)$  not suffer from localization and still contain  community information?}
\end{center}

\paragraph{Our contribution}
We take a first step toward explaining this phenomenon and theoretically justifying the Bethe-Hessian spectral method in \citep{saade2014spectral} when the expected degree of the SBM is constant or grows with $n$ arbitrarily slowly. Specifically, we establish the following results for the Bethe-Hessian matrix under the SBM:
\begin{itemize}
  \item When the expected degree $d\geq 2$, we show that above the Kesten-Stigum threshold, negative outliers of $H(\pm \sqrt{d})$ consistently estimate the number of communities in the SBM, which confirms a conjecture in \citep{saade2014spectral} for $d\geq 2$; see Theorem~\ref{thm:count}.
  \item When $d$ is sufficiently large, we can approximate the eigenvalue locations (see Theorem~\ref{thm:location_infty}) and eigenvectors (see Theorem~\ref{thm:eigenspace}) of the Bethe-Hessian above the Kesten-Stigum threshold. This characterization allows us to show that a spectral algorithm achieves weak recovery in the SBM; see Theorem~\ref{thm:weak_recovery}.
  \item When $d\to\infty$, the eigenvalue and eigenvector bounds of Theorems \ref{thm:location_infty} and \ref{thm:eigenspace} become even sharper, which allows our algorithm to achieve weak consistency (see Corollary \ref{cor:weak_consistency}). This provides a provable algorithm to achieve consistency without the degree regularization step (removing high-degree vertices).
  \item Along the way to proving these results, we also show that in the SBM, all informative outlier eigenvalues of the non-backtracking matrix are real and semisimple (see Theorem~\ref{thm:real_eigenvalues} and Lemma~\ref{lem:cluster_semisimplicity}). These properties have been empirically observed, but no justification exists in the previous literature.
\end{itemize}

To analyze the performance of the Bethe-Hessian spectral algorithms, we establish new connections between the non-backtracking matrix and the Bethe-Hessian matrix. Roughly speaking, the negative eigenvalues and eigenvectors of the Bethe-Hessian matrix approximately perform the function of a non-backtracking operator.

A key component in our argument is to show that the subspace spanned by the informative eigenvectors of $H(\pm \sqrt{d})$ can be approximated by the subspace spanned by eigenvectors of the $2n \times 2n$ reduced non-backtracking matrix. Therefore, results on the spectra of the non-backtracking matrix \citep{bordenave2018nonbacktracking,stephan2022sparse} can be used to approximate the eigenvalues and eigenvectors of the Bethe-Hessian $H(\pm \sqrt{d})$. These approximations are quantified using the Courant minimax principle, perturbation analysis of Hermitian matrices, and a continuity argument based on the Ihara-Bass formula. Additionally, we require a precise asymptotic analysis of the overlap between eigenvector components of the reduced non-backtracking matrix $\tilde{B}$, achieved through a quantitative version of local weak convergence for sparse random graphs studied in \citep{stephan2022sparse}.

\paragraph{Related work} There are only a few rigorous results for the Bethe-Hessian matrix in the literature. \citep{le2022estimating} showed that if the expected degree is $\omega(\log n)$ in an SBM, one can consistently estimate the number of assortative communities (those associated with positive eigenvalues of the expected adjacency matrix) and the locations of the negative eigenvalues of the Bethe-Hessian. \citep{hwang2023estimation} extended the range of consistent estimation down to $d = \omega(1)$; however, their proof technique does not provide any information on the location of the negative eigenvalues of the Bethe-Hessian. Very recently, \citep{mohanty2024robust} showed that above the Kesten-Stigum threshold, if the SBM has $k$ blocks and the probability matrix has $k-1$ repeated eigenvalues, then there are at least $k-1$ and at most $k$ negative eigenvalues of $H(\lambda)$ for a properly chosen $\lambda$. We provide a more detailed comparison with existing results in Section~\ref{sec:discuss}.

Different choices of the parameter $t$ for the Bethe-Hessian were proposed in \citet{dall2019revisiting,dall2021nishimori,dall2021unified}, where the authors empirically demonstrated improved performance on spectral clustering in the SBM compared to $H(\pm \sqrt{d})$. This is closely connected to an intriguing eigenvalue insider phenomenon \citep{dall2019revisiting,coste2021eigenvalues} for the non-backtracking matrix. Beyond community detection, the Bethe-Hessian matrix appears in different forms in various problems, including matrix completion \citep{saade2015matrix} and phase retrieval \citep{maillard2022construction}. A higher-order generalization of the Bethe-Hessian called the Kikuchi-Hessian \citep{saade2016spectral}, was introduced to study the computational-to-statistical gap in tensor PCA \citep{wein2019kikuchi}.

Our work provides a theoretical justification of the Bethe-Hessian method in the statistical physics literature. We expect this simple method can be adapted to develop new algorithms beyond the classical SBM (e.g., random graphs with random weights \citep{stephan2022non}, hypergraphs \citep{stephan2022sparse}, graphon estimation \citep{abbe2023learning}) and community detection with additional constraints (e.g., robustness \citep{stephan2019robustness,mohanty2024robust} and privacy \citep{mohamed2022differentially}). Moreover, we expect that this method can be adapted to provable and efficient algorithms for matrix and tensor completion with very few samples \citep{bordenave2020detection,stephan2024non}; an example of such a problem where the Bethe Hessian has shown promise can be found in \citep{saade2015matrix}.

\paragraph{Organization of the paper}
The rest of the paper is organized as follows. In Section~\ref{sec:prelim_model}, we define the stochastic block model and introduce model parameters. Section~\ref{sec:main} introduces the main results for the bounded and growing degree regimes followed by a discussion. Section~\ref{sec:prelim} collects preliminary results on $H(t)$ and the non-backtracking matrix. In Section~\ref{sec:negative}, we study negative eigenvalues of $H(t)$. The analysis of eigenvectors of $H(t)$ in the high-degree regime is given in Section~\ref{sec:high}. Section~\ref{sec:weak_recovery} contains proofs for weak recovery.
In Appendix~\ref{sec:app:perturbation}, we include auxiliary results for the spectral stability of Hermitian matrices. Appendix~\ref{sec:app:cor} includes the proof of Corollary~\ref{cor:hatd}.

\section{Preliminaries}\label{sec:prelim_model}

\paragraph{Stochastic Block Model}
We work under a general SBM setting as follows.
\begin{definition}[Stochastic block model (SBM)]\label{def:SBM}
  Let $P$ be an $r\times r$ symmetric nonnegative matrix, and let $\sigma:[n]\to[r]$ assign the vertices to $r$ nonempty blocks. Independently for each unordered pair of distinct vertices $i,j$, include the edge $\{i,j\}$ with probability $P_{\sigma(i),\sigma(j)}/n$.
\end{definition}

We have the following  model parameters:
\begin{itemize}
  \item Denote the block size ratio by $$\pi_k=\frac{|\{i\in [n], \sigma(i)=k \} |}{n}.$$
  \item We assume a common asymptotic expected degree:
    $$d:=\sum_{j\in [r]}P_{ij} \pi_j>1, \quad \forall i\in [r].$$
    When $d<1$, the random graph has no giant component, and detection is impossible.
  \item The signal matrix  $Q=P\Pi$, where $\Pi=\mathrm{diag}({\pi})$ is a diagonal matrix. Since $Q$ is similar to $\Pi^{1/2} P \Pi^{1/2}$, its eigenvalues are real and we order them by absolute value:
    $$|\mu_r|\leq |\mu_{r-1}|\leq \cdots \leq \mu_1=d.$$
  \item Let $\psi_i,1\leq i\leq r$ be the orthonormal eigenvectors of $\Pi^{1/2} P \Pi^{1/2}$. Eigenvectors of $Q$ are given by $\phi_i:=\Pi^{-1/2}\psi_i, 1\leq i\leq r$. Hence
    \begin{align}
      \langle \phi_i,\phi_j\rangle_{\pi}:=\sum_{k\in [r]}\pi_k\phi_i(k)\phi_j(k)=\langle \psi_i,\psi_j\rangle=\delta_{ij}.
    \end{align}
  \item The informative eigenvalues are defined by $\mu_i$ such that
    \begin{align}\label{eq:KS}
      \mu_i^2>d.
    \end{align}
    Let $r_0=|\{i\in[r]:\mu_i^2>d\}|$ be the number of informative eigenvalues; they are split between $r_+$ positive and $r_-$ negative eigenvalues, denoted
    \begin{align} \label{eq:mu_order}
      \mu_1^+ \geq \dots \geq \mu_{r_+}^+ > \sqrt{d} \quand \mu_1^- \leq \dots \leq \mu_{r_-}^- < -\sqrt{d} .
    \end{align}
  \item Define \[\tau_i=\frac{d}{\mu_i^2}\] to be the inverse signal-to-noise ratio associated with $\mu_i$, with the convention $\tau_i=+\infty$ if $\mu_i=0$. Similarly, set $\tau_i^{\pm}:=d/(\mu_i^{\pm})^2$. The Kesten-Stigum condition \eqref{eq:KS} implies that $\tau_i\in(0,1)$ for $i\in[r_0]$.
\end{itemize}

\paragraph{Overlap and recovery regimes} Given an estimate $\hat\sigma$ of the community assignment $\sigma$, the \emph{overlap} between the two vectors is defined as
\[ \operatorname{ov}(\sigma, \hat\sigma) = \max_{\mathfrak{p} \in \mathfrak{S}_r} \frac1n \sum_{x=1}^n \mathbf{1}\{\hat\sigma(x) = \mathfrak{p} \circ \sigma(x)\}, \]
where the maximum is over all permutations of $[r]$. Following the nomenclature of \citep{abbe2018community}, we say that an estimator $\hat\sigma$ (asymptotically) achieves:
\begin{itemize}
  \item \emph{weak recovery} if $\operatorname{ov}(\sigma,\hat\sigma)\ge\max_i\pi_i+\delta$ with high probability for some constant $\delta>0$,
  \item \emph{almost exact recovery} (or \emph{weak consistency}) if $\operatorname{ov}(\sigma, \hat\sigma) = 1 - o(1)$.
\end{itemize}
The baseline performance of $\max_i \pi_i$ corresponds to the dummy estimator that assigns every vertex to the largest community.

\paragraph{Non-backtracking matrix}
Let $G=(V,E)$ be a graph. For $V=[n]$, the $(i,j)$-th entry of the \textit{adjacency matrix} $A$ of a graph $G$ is defined as
\[ A_{ij}=
  \begin{cases}
    1  &  \text{if } \{i,j\}\in E,\\
    0 & \text{otherwise.}
\end{cases}\]
The degree matrix $D$ of a graph $G$ is a diagonal matrix where \[D_{ii}=\sum_{j\in V} A_{ij}.\]
Define the oriented edge set $\vec{E}$ for $G$ as
\begin{align} \label{eq:def_oriented}
  \vec{E}= \{ (i,j): \{ i,j\} \in E\}.
\end{align}
Each edge yields two oriented edges; therefore, $|\vec{E}| =2|E|$. The \textit{non-backtracking operator} $B$ of $G$ is a non-Hermitian operator of size $|\vec E| \times |\vec E|$. For any $(u,v), (x,y)\in \vec{E}$, the corresponding entry of $B$ is defined as $$
B_{(u,v),(x,y)} =
\begin{cases}
  1 &\text{$v=x$, $u \neq y$}, \\
  0 &\text{otherwise.}
\end{cases}
$$
A useful identity we will use in this paper is the following Ihara-Bass formula \citep{bass_iharaselberg_1992}:
\begin{lemma}[Ihara-Bass formula]\label{lem:Ihara-Bass} For any graph $G=(V,E)$ and any $z\in\mathbb C\setminus\{-1,1\}$, the following identity holds:
  \begin{align}\label{eq:Ihara}
    \det(B-zI)=(z^2-1)^{|E|-n} \det (z^2I -zA+D-I).
  \end{align}
  The identity extends to $z=\pm1$ after cancellation of common factors.
\end{lemma}

\paragraph{Reduced non-backtracking matrix and the Bethe-Hessian}

Consider the block matrix
\begin{align}\label{eq:tildeB}
  \tilde B=
  \begin{bmatrix}
    0 &  D-I\\
    -I  & A
  \end{bmatrix} \in \mathbb R^{2n\times 2n}.
\end{align}
Then from \eqref{eq:Ihara}, we have
\begin{align}\label{eq:det_tildB}
  \det(B-zI)=(z^2-1)^{|E|-n}\det (\tilde B-zI).
\end{align}
The identity \eqref{eq:det_tildB} implies that $B$ and $\tilde B$ share the same spectrum, up to the multiplicity of trivial eigenvalues $\pm 1$. We also call $\tilde B$ the \textit{reduced non-backtracking matrix} of $G$.

Let $v_i=
\begin{bmatrix}
  x_i\\
  y_i
\end{bmatrix}$ be an eigenvector of $\tilde B$ with associated eigenvalue $\lambda_i:=\lambda_i(\tilde B)\notin\{\pm1\}$, which is also an eigenvalue of $B$ and $x_i,y_i\in \dC^n$.  We normalize $y_i$ such that $\|y_i\|_2=1$.
By solving \[\tilde Bv_i=\lambda_i v_i,\] we obtain the following relation:
\begin{align}\label{eq:BHeq}
  (D-I)y_i&=\lambda_ix_i,\\
  -x_i+Ay_i&=\lambda_i y_i. \label{eq:Bheq1}
\end{align}
{In particular, $H(\lambda_i)y_i=0$, so $y_i$ is a unit vector in $\ker H(\lambda_i)$.}

When $G$ is a regular graph, one can show that $x_i,y_i$ are perfectly aligned \citep{zhu2023non}. As we will see in Section~\ref{sec:yi}, this is not the case for the SBM, but one can compute the asymptotic values of $\langle x_i,y_i\rangle$.
Spectral clustering with the eigenvectors of $\tilde B$ was first proposed in \citep{krzakala2013spectral}. This was rigorously justified in \citep{bordenave2018nonbacktracking}, with the correspondence made explicit in \citep{stephan2022sparse}, showing that $y_i, i\in [r_0]$ are correlated with eigenvectors of $Q$.

\paragraph{Notations} For $n\times n$ Hermitian matrices $A$, we order eigenvalues in decreasing order: $\lambda_1(A)\geq \cdots \geq \lambda_n(A)$. For any $n\times n$ non-Hermitian matrix $B$, we order its eigenvalues in decreasing order for the modulus: $|\lambda_1(B)|\geq \cdots \geq |\lambda_n(B)|$.

\section{Main results}\label{sec:main}

\subsection{Bounded expected degree regime} \label{sec:sparse}
In this section, we make the following assumption:
\begin{assumption}\label{assumption:fixed}
  The parameters $\pi_i, P_{ij}$ (and hence $d, r$ and $r_0$) are independent of $n$.
\end{assumption}
Our first result is an almost complete characterization of the relationship between the informative eigenvalues of $Q$ and the negative eigenvalues of $H(\pm\sqrt{d})$.
For $t\in\dR$ and $\eps > 0$, we define
$N_{\varepsilon}(t)$ to be the number of eigenvalues of $H(t)$ below $-\varepsilon$, and $N(t)$ to be the number of negative eigenvalues of $H(t)$.

\begin{theorem}[Estimating the number of communities]\label{thm:count}
  Assume that the average degree satisfies $d\geq 2$ and Assumption~\ref{assumption:fixed} holds. For any $\eps > 0$ small enough, the following holds with high probability for sufficiently large $n$:

  \begin{equation}
    N_{\eps}(\sqrt{d}) = N(\sqrt{d} + \eps) = r_+ \quand N_{\eps}(-\sqrt{d}) = N(-\sqrt{d} - \eps) = r_-.
  \end{equation}
 The same conclusion holds for $\varepsilon_n=\log^{-a}n$ for any  $a\in(0,1)$.
\end{theorem}

The total count in Theorem~\ref{thm:count} estimates $r_0=r_++r_-$; it equals the number $r$ of communities when all eigenvalues of $Q$ satisfy \eqref{eq:KS}.

Theorem~\ref{thm:count} rigorously justifies the prediction in \citep{saade2014spectral} for general stochastic block models when $d\geq 2$ for estimating the number of communities in the SBM, and extends the previous results from $d\to\infty$ in \citep{le2022estimating,hwang2023estimation} to fixed $d$. It remains open to show the consistency of the Bethe-Hessian estimator for a relatively narrow range of $d\in (1,2)$.

In fact,  the Bethe-Hessian method is parameter-free since $d$ can be estimated by the sample mean degree
\begin{align}\label{eq:hatd}
  \hat{d}:=\frac{1}{n}(d_1+\dots+d_n),
\end{align}
where $d_i$ is  the degree of vertex $i\in [n]$, which yields the following corollary:

\begin{corollary}\label{cor:hatd}
  Assume $d\ge2$ and Assumption~\ref{assumption:fixed}. Let $\varepsilon_n=\log^{-a}n$ for a fixed $a\in(0,1)$. For sufficiently large $n$, with high probability,
  \begin{align}
    N_{\varepsilon_n}\left(\sqrt{\hat d}\right)=N\left(\sqrt{\hat d}+\varepsilon_n\right)=r_+, \quad N_{\varepsilon_n}\left(-\sqrt{\hat d}\right)=N\left(-\sqrt{\hat d}-\varepsilon_n\right)=r_-.
  \end{align}
\end{corollary}

\subsubsection{Sketch of the proof for Theorem~\ref{thm:count}}

To lower bound the number of negative eigenvalues, we use the Courant-Fischer theorem by showing that $y^\top H(t) y < 0$ for any unit vector $y$ in a specific subspace. The dimension of the subspace provides a lower bound for the number of negative eigenvalues of $H(t)$ without finding their exact eigenvectors.
This subspace is given by known eigenvectors of $H(t)$ at different values of $t$; we take $y_i$ to be in the kernel of $H(\lambda_i)$, where $\lambda_i = \lambda_i(B)$ is the $i$-th eigenvalue of the non-backtracking matrix $B$, and the existence of $y_i$ is guaranteed by the Ihara-Bass formula \citep{bass_iharaselberg_1992}. This choice of $y_i$ allows us to compute exactly the quadratic form $y_i^\top H(t)y_i$:
\[ y_i^\top H(t)y_i = \left(t - \frac{\langle y_i, (D-I) y_i}{\lambda_i}\right)(t - \lambda_i), \]
where $D$ is the diagonal degree matrix. It remains to compute both roots of the above polynomial. Classical results on the non-backtracking operator \citep{bordenave2018nonbacktracking} show that $\lambda_i = \mu_i + o(1)$, and the scalar product $\langle y_i,(D-I)y_i\rangle$ can be computed using the local convergence properties of $G$ studied in \citep{stephan2022sparse}. We finally find
\[ y_i^\top H(t) y_i = \left(t - \frac{d + 1 - \tau_i}{\mu_i} \right)(t - \mu_i) + o(1) \quand y_i^\top H(t) y_j = o(1) \text{ for } i\not=j. \]
Whenever $d\ge2$, the first root has magnitude smaller than $\sqrt d$. Thus the quadratic form is negative at $t=\operatorname{sign}(\mu_i)\sqrt d$, providing the required subspace for each sign.
The upper bound is shown via a continuity argument,  similar to the one in \citep{mohanty2024robust}, which connects the negative eigenvalues of $H(t)$ to the real eigenvalue outliers of $B$. The complete proof of Theorem~\ref{thm:count} can be found in Appendix~\ref{sec:proof_count}.

The argument crucially rests on the fact that all outlier eigenvalues of $B$ are real with high probability. While generally accepted as folklore, all previous work on the non-backtracking spectrum of SBMs \citep{bordenave2018nonbacktracking,stephan2022non,stephan2022sparse} only proved that outlier eigenvalues are close to the real line.  {Lemma~\ref{lem:cluster_semisimplicity} closes this gap by proving that the informative outliers are real and semisimple, including when the corresponding eigenvalue of $Q$ is repeated.}
\begin{theorem}\label{thm:real_eigenvalues}
  Let $B$ be the non-backtracking operator of a stochastic block model under Assumption~\ref{assumption:fixed}. With high probability, the informative eigenvalues of $B$ and $\tilde B$ are real and semisimple, and can be labeled so that
  \[
    \lambda_i=\mu_i+O(n^{-c}),\qquad i\in[r_0],
  \]
  for some constant $c>0$. The associated eigenvectors $\binom{x_i}{y_i}$ of $\tilde B$ can be chosen to form a real eigenbasis for the informative eigenspaces.
\end{theorem}

{ When $d=\omega(\log n)$, reality of these eigenvalues has been justified} in \citep{le2022estimating,coste2021eigenvalues,wang2017limiting} based on a deterministic result in \citet[Theorem 3.7]{angel2015non}.

\subsection{Growing degree regime}
We allow the number of communities and the model parameters to vary with $n$, subject to the following uniform control. These hypotheses ensure that the local tree approximations and the perturbation estimates used below hold uniformly.
\begin{assumption}[Uniform parameter control]\label{assumption:uniform}
  For some fixed $K>0$, $d\ge2$ and
  \begin{equation}\label{eq:regime_d}
    \max\{r,d,p_{\max},\pi_{\min}^{-1},\Delta^{-1}\}\le(\log n)^K,
  \end{equation}
  where $p_{\max}=\max_{a,b}P_{ab}$, $\pi_{\min}=\min_a\pi_a$, and
  \[
    \Delta=\min\{|\mu_i-\mu_j|:i,j\le r_0,\ \mu_i\ne\mu_j\}.
  \]
  If there is only one distinct informative eigenvalue, set $\Delta=1$.
\end{assumption}
The hypotheses include fixed block proportions with $P=\alpha_nP_0$ for a fixed matrix $P_0$, provided the informative values satisfy this separation and threshold condition. They also allow $r$ to grow.

\subsubsection{Spectrum of \texorpdfstring{$H(\pm \sqrt{d})$}{H(plus or minus sqrt(d))}}
The first theorem establishes the concentration of the negative outlier locations when $d\to\infty$:

\begin{theorem}[Eigenvalue locations] \label{thm:location_infty}
  Assume Assumption~\ref{assumption:uniform} and $\tau_{r_0}\le c<1$, where $c$ is independent of $n$. For $d\ge d_c$, where $d_c>0$ is a sufficiently large constant depending only on $c$, the following holds with high probability:
  \begin{enumerate}
    \item  For $1\leq k\leq r_+$,
      \begin{align}
        \lambda_{n-k+1}(H(\sqrt{d}))= (\sqrt{d} - \mu_k^+)\left(\sqrt{d} - \frac{d}{\mu_k^+}\right) + O(\sqrt d).
      \end{align}
    \item  For $1\leq k\leq r_-$,
      \begin{align}
        \lambda_{n-k+1}(H(-\sqrt{d}))=(-\sqrt{d} - \mu_{k}^-)\left(-\sqrt{d} - \frac{d}{\mu_{k}^-}\right) + O(\sqrt d).
      \end{align}
    \item In particular, if $d\to\infty$, then for $k\in[r_+]$,
      \begin{align}
        \lambda_{n-k+1}\left(\frac{1}{d}H(\sqrt{d})\right)=\left(1-\frac{1}{\sqrt{\tau_k^+}}\right)\left(1-\sqrt{\tau_k^+}\right)+o(1).
      \end{align}
      And for $k\in [r_-]$,
      \[\lambda_{n-k+1}\left(\frac{1}{d}H(-\sqrt{d})\right)=\left(1-\frac{1}{\sqrt{\tau_k^-}}\right)\left(1-\sqrt{\tau_k^-}\right)+o(1).\]
  \end{enumerate}
\end{theorem}

Theorem~\ref{thm:location_infty} is proved in Proposition~\ref{prop:eigendecomposition_H} (i). We can also obtain a related result on the eigenvectors of $H(t)$. For $k\in[r_\pm]$, use the real unit vertex components $y_k^\pm$ of the eigenvectors selected in Proposition~\ref{prop:eigen_weak_convergence}. They satisfy $H(\lambda_k^\pm)y_k^\pm=0$, where $\lambda_k^+$ is the $k$-th largest positive informative eigenvalue of $B$ and $\lambda_k^-$ is the $k$-th smallest negative informative eigenvalue. Lemma~\ref{lem:j_semisimple} ensures that these eigenvalues are real.

\begin{theorem}[Eigenvector approximation]\label{thm:eigenspace}
  Assume Assumption~\ref{assumption:uniform}, $\tau_{r_0}\le c<1$, and $d\ge d_c$. Let $v_1^\pm,\ldots,v_{r_\pm}^\pm$ be orthonormal eigenvectors of $H(\pm\sqrt d)$ corresponding to its $r_\pm$ smallest eigenvalues.
  Then, with high probability, there exist orthogonal matrices $O^+ \in \dR^{r_+ \times r_+}, O^- \in \dR^{r_- \times r_-}$ such that
  \begin{align*}
    \norm{V^+ - Y^+ O^+ }_F &= O(\sqrt{r_+/d}),\\
    \norm{V^- - Y^- O^- }_F &= O(\sqrt{r_-/d}),
  \end{align*}
  where $Y^+\in \dR^{n\times r_+}$ is the matrix whose columns are the $y_k^+$.
  The matrices $V^+, Y^-$, and $V^-$ are defined similarly, with columns $v_k^+$, $y_k^-$, and $v_k^-$, respectively.
\end{theorem}

Theorem~\ref{thm:eigenspace} is proved in Proposition~\ref{prop:eigendecomposition_H} (ii).
Since the overlaps between $y_i^{\pm}$ and $\tilde\phi_i^\pm=\phi_i^\pm\circ\sigma$ are known from \citep{stephan2022sparse}, we easily get the following corollary, whose proof is given by Lemma~\ref{lem:bethe_overlaps_sharp}.
\begin{corollary}\label{cor:bethe_overlaps}
  Under the same assumptions as in Theorem~\ref{thm:eigenspace}, with high probability there exist orthogonal matrices $\tilde O^+ \in \dR^{r_+ \times r_+}, \tilde O^- \in \dR^{r_- \times r_-}$ such that
  \begin{align*}
    \norm{V^+ - \check\Phi^+ \tilde O^+ }_F
      &\leq 2\left(\sum_{k=1}^{r_+}\tau_k^+\right)^{1/2}+O(\sqrt{r_+/d}),\\
    \norm{V^- - \check\Phi^- \tilde O^-}_F
      &\leq 2\left(\sum_{k=1}^{r_-}\tau_k^-\right)^{1/2}+O(\sqrt{r_-/d}).
  \end{align*}
  Here $\check\Phi^\pm_{xk}=n^{-1/2}\phi_k^\pm(\sigma(x))$, where $\phi_k^\pm$ is a $\pi$-orthonormal eigenvector of $Q$ associated with $\mu_k^\pm$.
\end{corollary}

\subsubsection{Weak recovery and consistency for spectral algorithms}
Let $\mathbb{M}_{n,r}$ be the collection of all $n\times r$ matrices where each row has exactly one $1$ and $(r-1)$ zeros. Given a spectral embedding matrix $V\in \dR^{n\times r}$, the  $k$-means clustering problem can be defined as
\begin{align}
  (\hat{\Sigma},\hat{P})\in\operatorname*{argmin}_{\Sigma\in \mathbb M_{n,r}, P\in \dR^{r\times r}} \| \Sigma P - V\|_F^2,
\end{align}
where $\hat{\Sigma}$ represents a partition of $n$ data points into  $r$ clusters.
For fixed $r$ and $\eta$, a randomized $(1+\eta)$-approximation can be computed in time linear in the input size \citep{kumar2004simple}. The statistical guarantees below apply to any pair $(\hat\Sigma,\hat P)$ satisfying
\begin{align}\label{eq:approx_k_means_eta}
  \| \hat \Sigma \hat P -V\|_F^2\leq (1+\eta)\min_{\Sigma\in \mathbb M_{n,r}, P\in \dR^{r\times r }} \| \Sigma P -V \|_F^2.
\end{align}

\begin{algorithm}
  \caption{Spectral clustering with the Bethe-Hessian}
  \label{alg:spectral_clust}
  \begin{algorithmic}[1]
    \REQUIRE Adjacency matrix $A$ of a graph $G$, approximation parameter $\eta>0$
    \STATE Let $\hat d$ be the average degree of $G$. Construct two Bethe-Hessian matrices $H\left(\pm \sqrt{\hat{d}}\right)$.
    \STATE Set $\widehat\zeta_n=\hat d/\sqrt{\log n}$. Count the eigenvalues below $-\widehat\zeta_n$ of $H\left(\pm\sqrt{\hat d}\right)$ and use these counts as estimates of $r_+$ and $r_-$, respectively. Let $r=r_{-}+r_+$.
    \STATE Compute orthonormal eigenvectors $v_k^\pm$, $k\in[r_\pm]$, for eigenvalues of $H\left(\pm \sqrt{\hat{d}}\right)$ below $-\widehat\zeta_n$.
    \STATE Let $V^{\pm}\in \mathbb R^{n\times r_{\pm}}$ be the matrix whose columns are $v_k^{\pm}$ for $k\in [r_{\pm}]$ and $V=[V^+,V^-]\in \dR^{n\times r}$.
    \STATE Let $(\hat \Sigma, \hat{P})$ be a $(1+\eta)$-approximate solution to the $k$-means problem \eqref{eq:approx_k_means_eta}.
    \STATE For each $i\in [n]$, let $\hat{\sigma}_i=\sum_{k=1}^{r} k \cdot \mathbf{1} \{\hat{\Sigma}_{i,k}=1\}$.
    \RETURN $\hat{\sigma}$.
  \end{algorithmic}
\end{algorithm}

Since $\check\Phi^\pm_{xk}=n^{-1/2}\phi_k^\pm(\sigma(x))$, vertices with the same label have identical rows in $\check\Phi^\pm$. Corollary \ref{cor:bethe_overlaps} bounds the Frobenius distance between $V^\pm$ and an orthogonal transform of $\check\Phi^\pm$. The following theorem quantifies the overlap achieved by Algorithm~\ref{alg:spectral_clust} for large expected degree:

\begin{theorem}[Weak recovery at large degree]\label{thm:weak_recovery}
  Assume Assumption~\ref{assumption:uniform}, $\max_{i\le r}\tau_i\le c<1$, and $d\ge d_c$. Let $\hat\sigma$ be the output of Algorithm \ref{alg:spectral_clust} on $G$. Then for any constant $\eta>0$, there exists a constant $C=C(c,\eta)>0$ such that with high probability,
  \begin{equation}
    \mathrm{ov}(\sigma,\hat\sigma)\ge 1-C\sum_{i=1}^r\tau_i.
  \end{equation}
\end{theorem}
This bound guarantees weak recovery when $1-C\sum_{i=1}^r\tau_i-\max_i\pi_i$ is bounded below by a positive constant.
The proof of Theorem~\ref{thm:weak_recovery} is given in Appendix~\ref{sec:weak_recovery}.
{This provides a simple spectral algorithm without degree regularization for weak recovery. Since $|\mu_i|\le d$, we have $\tau_i\ge1/d$ and hence $r/d\le\sum_{i=1}^r\tau_i$. Thus $\sum_{i=1}^r\tau_i=o(1)$ already implies $r/d=o(1)$ and $d\to\infty$, giving weak consistency:}

\begin{corollary}[Weak consistency]\label{cor:weak_consistency}
  Assume Assumption~\ref{assumption:uniform} and $\sum_{i=1}^r\tau_i=o(1)$.
  Then for any constant $\eta>0$, with high probability, the output $\hat\sigma$ of Algorithm \ref{alg:spectral_clust} satisfies
  \[ \mathrm{ov}(\sigma, \hat\sigma) = 1 - o(1).\]
\end{corollary}

\subsection{Discussion}\label{sec:discuss}

\paragraph{Comparison with \citep{mohanty2024robust}}
The most recent result on the spectrum of the Bethe-Hessian matrix comes from \citep{mohanty2024robust}, which also establishes a relation between the eigenvalues of $Q$ and the negative outliers of $H(t)$ for $d>1$ with a choice of $t>\sqrt{d}$ depending on $\mu_2$. However, \citet[Theorem 5.1]{mohanty2024robust} actually only addresses a single eigenspace of $H(t)$, and therefore their Proposition 2.1 is only valid for SBMs with symmetric communities of equal sizes and a parameter $t$ that depends on the eigenvalue $\mu_2$. Our result covers a more general class of SBMs, and we can exactly count the informative directions for $d\ge2$, and hence the communities when $r_0=r$. We can also retrieve the results of \citep{mohanty2024robust} for any $d \geq 1$ with a more careful choice of $t$.

\paragraph{Comparison with \citep{le2022estimating,hwang2023estimation}} Both \citep{le2022estimating} and \citep{hwang2023estimation} only considered estimating the number of communities with $H(\sqrt{d})$ in the assortative cases where the expected adjacency matrix has nonnegative eigenvalues, while the original conjecture in \citep{saade2014spectral} is for both assortative and disassortative ones by using $H(\pm \sqrt d)$. We do not make the assortative assumption, and our result holds for a more general SBM setting. \citep{le2022estimating} demonstrates the consistency of the Bethe-Hessian method for estimating the number of communities when $d = \omega(\log n)$, which can be shown through matrix norm concentration and degree concentration under this regime \citep{le2017concentration,benaych2020spectral}. \citep{hwang2023estimation} relates the Bethe-Hessian spectrum to a normalized Laplacian with regularization studied in \citep{le2017concentration}  using Sylvester’s Law of Inertia. However, their method is limited to the case $d \to \infty$ and does not provide information on the locations of eigenvalues. Our work overcomes these limitations by exploring a novel connection to the eigenvalues and eigenvectors of the reduced non-backtracking matrix studied in \citep{stephan2022sparse} down to the bounded expected degree regime.

\paragraph{Spectral algorithms for weak consistency}
Most existing methods to achieve weak consistency in the case $d\to\infty$ require a regularization step \citep{abbe2018community,zhang2024fundamental} to remove high-degree vertices.
In the 2-block case, weak consistency without removing high-degree vertices was developed in \citep{le2017concentration} based on regularized Laplacian matrices. The Bethe-Hessian approach is parameter-free, and it does not need additional regularization on high-degree vertices.

\paragraph{Extension to weighted graphs} As discussed in Section~\ref{sec:intro}, reducing the non-backtracking operator matrix \(B\) to its \(2n \times 2n\) form \(\tilde{B}\) only applies to unweighted graphs, therefore making the non-backtracking method inefficient for clustering weighted graphs. However, for the Bethe--Hessian matrix, a straightforward extension to weighted graphs was presented in~\citep{saade2014spectral}. Assume each edge $(i,j)$ has a weight $w_{ij}$, then we can define $H(t)$ such that
\[ H(t)_{ij}=\delta_{ij} \left(1+\sum_{k\in \partial_i} \frac{w_{ik}^2}{t^2-w_{ik}^2} \right)-\frac{tw_{ij}A_{ij}}{t^2-w_{ij}^2},\]
where $\partial_i$ denotes the set of all neighbors of vertex $i$.
When all weights are $1$, this is the same as the unweighted Bethe-Hessian matrix up to a normalization factor $\frac{1}{t^2-1}$. Such a generalization might be useful for developing more efficient spectral methods for matrix completion in the very sparse regime \citep{bordenave2020detection}.

\acks{Y.Z. was partially supported by NSF-Simons Research
  Collaborations on the Mathematical and Scientific Foundations of Deep Learning  and an AMS-Simons Travel
  Grant. This material is based upon work supported by the Swedish Research Council under grant
no. 2021-06594 while Y.Z. was in residence at Institut Mittag-Leffler in Djursholm, Sweden during the Fall of 2024.}

\bibliography{ref.bib}

\appendix

\section{Preliminary results}\label{sec:prelim}

We work under Assumption~\ref{assumption:fixed} or Assumption~\ref{assumption:uniform}, and assume that
\[
  \max_{i\le r_0}\tau_i\le c,\qquad c\in(0,1).
\]
With $p_{\max}$, $\pi_{\min}$, and $\Delta$ as in Assumption~\ref{assumption:uniform}, define
\begin{align}\label{eq:def_epsn}
 L_n&=\max\{r,d,p_{\max},\pi_{\min}^{-1},\Delta^{-1}\},\qquad
 \ell=\left\lfloor\kappa\frac{\log n}{\log L_n}\right\rfloor,\notag\\
 \epsilon_n&=\exp\left(-\alpha\frac{\log n}{\log L_n}\right).
\end{align}
We first choose $\kappa>0$ and then $\alpha>0$ sufficiently small, depending only on $c$. The estimates below hold simultaneously with probability at least $1-O(\epsilon_n)$. Under Assumption~\ref{assumption:uniform}, \eqref{eq:regime_d} gives
\[
 2\le L_n\le(\log n)^K,\qquad
 \log 2\le\log L_n\le K\log\log n.
\]
Consequently,
\[
 \epsilon_n\le\exp\left(-\frac{\alpha}{K}\frac{\log n}{\log\log n}\right).
\]
For fixed parameters, $L_n=L$ is constant and $\epsilon_n=n^{-\alpha/\log L}$.

%

\subsection{Outlier eigenvalues in the non-backtracking operator}

The results in \citep{stephan2022sparse} locate the informative outliers of $B$ and approximate their eigenvectors by the corresponding pseudo-eigenspaces. We use these estimates to prove that the informative outliers are real and semisimple, including when an eigenvalue of $Q$ is repeated.

Define $J\in\dR^{|\vec E|\times|\vec E|}$ by
\[
 J_{(u,v),(a,b)}=\mathbf{1}\{u=b,\ v=a\},
\]
for $(u,v),(a,b)\in\vec E$.
Then $J^2=I$, $J^*=J$, and $B^*J=JB$.
\begin{lemma}[Quantitative input from \citep{stephan2022sparse}]\label{lem:SZ_precise_inputs}
 Let $\tilde\phi_i(x)=\phi_i(\sigma(x))$, and define $\chi_i$ on directed edges by $\chi_i(u,v)=\tilde\phi_i(u)$. Set
  \[
    u_i=\frac{B^\ell J\chi_i}{\mu_i^{\ell+1}\sqrt n},\qquad
    v_i=\frac{(B^*)^\ell\chi_i}{\mu_i^\ell\sqrt n},\qquad i\le r_0.
  \]
  Set $U=[u_1,\dots,u_{r_0}]$ and $V=[v_1,\dots, v_{r_0}]$. The following hold with high probability:
  \begin{enumerate}
    \item We have
    \[
      \|U^*V-I\|=O(\epsilon_n),\qquad
      \|U^*U-\operatorname{diag}(\gamma_{i,U}^{(\ell)}:i\le r_0)\|=O(\epsilon_n),
    \]
    where
    \[
      \gamma_{i,U}^{(\ell)}
       =\frac d{\mu_i^2}\frac{1-\tau_i^{\ell+1}}{1-\tau_i},
      \qquad \frac1d\le\gamma_{i,U}^{(\ell)}\le(1-c)^{-1}.
    \]
    \item There is a permutation $\mathfrak p$ of the informative indices such that
    \[
       |\lambda_{\mathfrak p(i)}(B)-\mu_i|=O(\epsilon_n),\qquad i\le r_0.
    \]
    Every remaining eigenvalue has modulus at most
    \begin{equation}\label{eq:bulk_scales}
      \sqrt d\left(1+C\frac{\log L_n\,\log\log n}{\log n}\right),
    \end{equation}
    for a sufficiently large constant $C$ independent of $n$.
    If $\xi$ is any unit eigenvector for $\lambda_{\mathfrak p(i)}(B)$, then there is a unit vector $\zeta\in\operatorname{span}\{u_j:\mu_j=\mu_i\}$ with
    \[
       \|\xi-\zeta\|=O(\epsilon_n).
    \]
    \item For each $i\le r_0$, let $\binom{x_i}{y_i}$ be an eigenvector of $\tilde B$ associated with the eigenvalue $\lambda_{\mathfrak p(i)}(B)$ from (ii). There is an eigenvector $\phi_i'$ of $Q$ with eigenvalue $\mu_i$ such that the lower block $y_i\in\dC^n$, indexed by the vertices of $G$, satisfies
    \[
      \frac{\langle y_i,\tilde\phi_i'\rangle}
           {\|y_i\|\,\|\tilde\phi_i'\|}
       =\sqrt{1-\tau_i}+O(\epsilon_n),\qquad
      \tilde\phi_i'(x)=\phi_i'(\sigma(x)).
    \]
  \end{enumerate}
\end{lemma}
\begin{proof}
Under Assumption~\ref{assumption:fixed}, the conclusions follow from \citep[Proposition~1, Lemma~4, and Theorem~2]{stephan2022sparse}. Under Assumption~\ref{assumption:uniform}, repeating these proofs with careful bookkeeping of the dependence on the model parameters yields the stated bounds in the growing-parameter regime.
\end{proof}

We apply Lemma~\ref{lem:SZ_precise_inputs}(ii) to the  cluster associated with a repeated informative eigenvalue of $Q$.

\begin{lemma}\label{lem:cluster_count_and_control}
  Fix an informative eigenvalue $\mu$ of $Q$, and let
  \[
    I_\mu=\{i\in [r_0]:\mu_i=\mu\}.
  \]
  Set
  \[
    \rho_\mu=\frac14\min\{\Delta,|\mu|-\sqrt d\}.
  \]
  Under Assumption~\ref{assumption:uniform}, we also have $\rho_\mu\ge C^{-1}(\log n)^{-K}$ for a constant $C=C(c)>0$, so $\epsilon_n=o(\rho_\mu)$.
  Let $\mathsf U_\mu$ be the block of pseudo-eigenvectors $u_j$ from Lemma~\ref{lem:SZ_precise_inputs}{} with indices $j\in I_\mu$, and set
  \[
    W_\mu=\operatorname{span}\{u_j:j\in I_\mu\}.
  \]
  With high probability, the disk $D(\mu,\rho_\mu)$ contains exactly $|I_\mu|$ eigenvalues of $B$, counted with algebraic multiplicity. Moreover, if $\xi$ is any unit eigenvector of $B$ whose eigenvalue belongs to $D(\mu,\rho_\mu)$, then there exists a unit vector $\zeta\in W_\mu$ such that
  \[
    \|\xi-\zeta\|=O(\epsilon_n).
  \]
\end{lemma}

\begin{proof}
  Apply Lemma~\ref{lem:SZ_precise_inputs}(ii), with the permutation $\mathfrak p$ specified there. Then
  \[
    |\lambda_{\mathfrak p(i)}(B)-\mu_i|=O(\epsilon_n)=o(\rho_\mu),\qquad i\in I_\mu,
  \]
  so these eigenvalues lie in $D(\mu,\rho_\mu)$ for all sufficiently large $n$. For every other informative index $j\in[r_0]\setminus I_\mu$, the same eigenvalue estimate gives
  \begin{align*}
    |\lambda_{\mathfrak p(j)}(B)-\mu|
    &\ge|\mu_j-\mu|-|\lambda_{\mathfrak p(j)}(B)-\mu_j|\\
    &\ge\Delta-O(\epsilon_n)\ge\frac{\Delta}{2}\ge2\rho_\mu,
  \end{align*}
  where we used $\epsilon_n=o(\rho_\mu)$ and $\rho_\mu\le\Delta/4$.
  Finally, $\tau_i\le c<1$ gives $|\mu|-\sqrt d\ge\sqrt d(c^{-1/2}-1)$. Since $\log L_n\,\log\log n/\log n=o(1)$, every remaining eigenvalue $\lambda$ satisfies, by \eqref{eq:bulk_scales},
  \begin{align*}
    |\lambda-\mu|
    &\ge|\mu|-\sqrt d-C\sqrt d\,\frac{\log L_n\,\log\log n}{\log n}\ge\frac12(|\mu|-\sqrt d)\ge2\rho_\mu
  \end{align*}
  for all sufficiently large $n$. Thus the only eigenvalues in $D(\mu,\rho_\mu)$ are $\lambda_{\mathfrak p(i)}(B)$ with $i\in I_\mu$, giving exactly $|I_\mu|$ eigenvalues counted with algebraic multiplicity.

  Let $\xi$ be any unit eigenvector whose eigenvalue $\lambda$ belongs to $D(\mu,\rho_\mu)$. The algebraic count gives $\lambda=\lambda_{\mathfrak p(i)}(B)$ for some $i\in I_\mu$. Every such index corresponds to the same eigenvalue $\mu$ of $Q$. Lemma~\ref{lem:SZ_precise_inputs}(ii) therefore applies to $\xi$ and gives a unit vector
    $\zeta\in \operatorname{span}\{u_j:\mu_j=\mu_i\}=W_\mu$
  with
    $\|\xi-\zeta\|=O(\epsilon_n)$.
\end{proof}

Denote
\[[f,g]_J=g^*Jf\] for the  Hermitian form associated with the matrix $J$.
We next show that the $J$-form is definite on the pseudo-eigenspace associated with each informative eigenvalue of $Q$.


\begin{lemma}\label{lem:pseudosignal_definite}
  On the high-probability event in  Lemma~\ref{lem:SZ_precise_inputs}, for all sufficiently large $n$ we have
  \[
    \operatorname{sign}(\mu)[\zeta,\zeta]_J\geq\frac{1-c}{4d}\|\zeta\|^2
  \]
  for every distinct informative eigenvalue $\mu$ of $Q$ and every vector $\zeta\in W_\mu=\operatorname{span}\{u_j:j\in I_\mu\}$.
\end{lemma}

\begin{proof}
Let $\mathsf U_\mu,\mathsf V_\mu$ be the $U,V$ matrices from Lemma~\ref{lem:SZ_precise_inputs}, restricted to $I_\mu$. Their definitions and $J(B^*)^\ell=B^\ell J$ give $J\mathsf U_\mu=\mu^{-1}\mathsf V_\mu$. For $\zeta=\mathsf U_\mu a$, Lemma~\ref{lem:SZ_precise_inputs}(i) therefore gives, for all sufficiently large $n$,
\begin{align*}
 \operatorname{sign}(\mu)[\zeta,\zeta]_J
   &=\frac{a^*\mathsf U_\mu^*\mathsf V_\mu a}{|\mu|}
     \ge\frac{\|a\|^2}{2|\mu|},\\
 \|\zeta\|^2
   &=a^*\mathsf U_\mu^*\mathsf U_\mu a
     \le\frac{2}{1-c}\|a\|^2.
\end{align*}
Combining these bounds and $|\mu|\le d$ yields
\[
 \operatorname{sign}(\mu)[\zeta,\zeta]_J
 \ge\frac{1-c}{4|\mu|}\|\zeta\|^2
 \ge\frac{1-c}{4d}\|\zeta\|^2.
\]
\end{proof}

\begin{lemma}[$J$-form on the outlier eigenvectors]\label{lem:actual_form_bound}
  With high probability, for every distinct informative eigenvalue $\mu$ and every unit eigenvector $\xi$ of $B$ whose eigenvalue belongs to $D(\mu,\rho_\mu)$,
  \[
    \operatorname{sign}(\mu)[\xi,\xi]_J\geq\frac{1-c}{8d}>0.
  \]
\end{lemma}

\begin{proof}
  By Lemma~\ref{lem:cluster_count_and_control}, there is a unit vector $\zeta\in W_\mu$ such that $\|\xi-\zeta\|=O(\epsilon_n)=o(d^{-1})$. Since $\|J\|=1$,
  \begin{align}\label{eq:J_ineq}
    |[\xi,\xi]_J-[\zeta,\zeta]_J|
      \leq |[\xi-\zeta,\xi]_J|+|[\zeta,\xi-\zeta]_J|
      \leq2\|\xi-\zeta\|=O(\epsilon_n)=o(d^{-1}).
  \end{align}
  Lemma~\ref{lem:pseudosignal_definite}{} gives $\operatorname{sign}(\mu)[\zeta,\zeta]_J\geq(1-c)/(4d)$ because $\zeta$ is a unit vector. Then~\eqref{eq:J_ineq} gives the desired bound.
\end{proof}

\begin{lemma}[Outlier eigenvalues of $B$ are real and semisimple]\label{lem:j_semisimple}
  With high probability, for every distinct informative eigenvalue $\mu$ of $Q$, every eigenvalue of $B$ in $D(\mu,\rho_\mu)$ is real and semisimple, where $\rho_\mu$ is chosen as in Lemma~\ref{lem:cluster_count_and_control}.
\end{lemma}

\begin{proof}
  Work on the high-probability event of Lemma~\ref{lem:actual_form_bound}. Fix an informative eigenvalue $\mu$ of $Q$, and let $B\xi=\lambda\xi$ with $\lambda\in D(\mu,\rho_\mu)$ and $\xi\neq0$. Applying that lemma to $\xi/\|\xi\|$ gives $[\xi,\xi]_J\neq0$. Since $B^*J=JB$, we have $[Bf,g]_J=[f,Bg]_J$ for all $f,g$. Therefore
  \[
    \lambda[\xi,\xi]_J=[B\xi,\xi]_J=[\xi,B\xi]_J=\overline{\lambda}[\xi,\xi]_J .
  \]
  Since $[\xi,\xi]_J\neq0$, we obtain $\lambda=\overline{\lambda}$.
  Now suppose that $\lambda$ has a Jordan chain of length at least two. Then there are vectors $\eta_0\neq0$ and $\eta_1$ such that
  \[
    (B-\lambda I)\eta_1=\eta_0,\qquad (B-\lambda I)\eta_0=0.
  \]
  Since $\lambda$ is real and $B-\lambda I$ is also self-adjoint for $[\cdot,\cdot]_J$,
  \[
    [\eta_0,\eta_0]_J=[(B-\lambda I)\eta_1,\eta_0]_J
      =[\eta_1,(B-\lambda I)\eta_0]_J=0,
  \]
  contradicting Lemma~\ref{lem:actual_form_bound}. Hence no Jordan chain of length at least two exists, so $\lambda$ is semisimple.
\end{proof}

We transfer this conclusion to $\tilde B$ by showing that both algebraic and geometric multiplicities agree away from $\pm1$.

\begin{lemma}[Semisimplicity of repeated outlier clusters]\label{lem:cluster_semisimplicity}
  Fix an informative eigenvalue $\mu$ of $Q$.  For the radius $\rho_\mu$ chosen in Lemma~\ref{lem:cluster_count_and_control}, with high probability,
  \begin{align}
    \sum_{\lambda\in \operatorname{spec}(\tilde B)\cap D(\mu,\rho_\mu)}
    \operatorname{algmult}_{\tilde B}(\lambda)
    =
    \sum_{\lambda\in \operatorname{spec}(\tilde B)\cap D(\mu,\rho_\mu)}
    \operatorname{geommult}_{\tilde B}(\lambda)
    =|I_\mu| . \label{eq:cluster_geometric_count}
  \end{align}
  Moreover, every eigenvalue of $\tilde B$ in $D(\mu,\rho_\mu)$ is real and semisimple. In particular, one may choose
  $|I_\mu|$ real eigenvectors of $\tilde B$ whose eigenvalues lie in this disk.
\end{lemma}

\begin{proof}
By Lemmas~\ref{lem:cluster_count_and_control} and~\ref{lem:j_semisimple}, with high probability the disk $D(\mu,\rho_\mu)$ contains exactly $|I_\mu|$ eigenvalues of $B$, counted with algebraic multiplicity, and all are real and semisimple. Since $|\mu|-\rho_\mu>\sqrt d>1$, the disk excludes $\pm1$. The determinant identity \eqref{eq:det_tildB} therefore gives the same eigenvalues and algebraic multiplicities for $\tilde B$ in this disk.

Fix one of these eigenvalues $\lambda$. For $f\in\ker(B-\lambda I)$, set
\[
 y_u=\sum_{v:\{u,v\}\in E}f_{(u,v)},\qquad x=Ay-\lambda y.
\]
The relation $Bf=\lambda f$ gives $\lambda f_{(u,v)}=y_v-f_{(v,u)}$, and hence
\[
 (\lambda^2-1)f_{(u,v)}=\lambda y_v-y_u.
\]
Summing over $v$ adjacent to $u$ yields $H(\lambda)y=0$, so $(D-I)y=\lambda x$ and $-x+Ay=\lambda y$. Since $\lambda^2\ne1$, the displayed edge identity also shows that $y=0$ implies $f=0$. Thus the linear map $f\mapsto\binom{x}{y}$ from $\ker(B-\lambda I)$ to $\ker(\tilde B-\lambda I)$ is injective. Consequently,
\begin{align*}
 \dim\ker(\tilde B-\lambda I)
 &\ge\dim\ker(B-\lambda I)=\operatorname{algmult}_B(\lambda)\\
 &=\operatorname{algmult}_{\tilde B}(\lambda)
   \ge\dim\ker(\tilde B-\lambda I).
\end{align*}
Equality holds throughout. Hence $\lambda$ is semisimple for $\tilde B$, and this transfer map is bijective. Summing over the eigenvalues in the disk proves \eqref{eq:cluster_geometric_count}. Since $\tilde B$ is real and these eigenvalues are real, their eigenspaces have real bases.
\end{proof}

\begin{proof}[Proof of Theorem~\ref{thm:real_eigenvalues}]
  Apply Lemma~\ref{lem:cluster_semisimplicity}{} to each distinct informative eigenvalue $\mu$ of $Q$. The disks $D(\mu,\rho_\mu)$ are disjoint and contain exactly the $r_0$ informative outliers of $\tilde B$. In the disk associated with $\mu$, the lemma gives $|I_\mu|$ real semisimple eigenvalues and a real eigenbasis. Label these eigenvalues by the indices $i$ such that $\mu_i=\mu$. Lemma~\ref{lem:SZ_precise_inputs}(ii) then gives $\lambda_i=\mu_i+O(\epsilon_n)$. Since this error is polynomially small in the fixed parameter regime, the theorem follows.
\end{proof}

Within each informative cluster, choose a real eigenbasis of $B$ whose vectors are mutually $J$-orthogonal. This is possible by Lemmas~\ref{lem:actual_form_bound} and~\ref{lem:j_semisimple}: the form $\operatorname{sign}(\mu)[\cdot,\cdot]_J$ is positive definite on each eigenspace, and distinct eigenspaces are $J$-orthogonal. Transfer this basis to $\tilde B$ by the map in the proof of Lemma~\ref{lem:cluster_semisimplicity}, and normalize so that $\|y_i\|=1$. Proposition~\ref{prop:eigen_weak_convergence} below matches this choice with a $\pi$-orthonormal eigenbasis of $Q$.

\subsection{Properties of the vectors \texorpdfstring{$y_i$}{y i}}\label{sec:yi}

The results of \citet{bordenave2018nonbacktracking, stephan2022sparse} link the eigenvectors $y_i$ to specific processes on Galton-Watson trees; this allows us to compute many properties of the $y_i$ directly. The next proposition records these estimates with the uniform parameter bounds and simultaneous choice of bases used here.

\begin{proposition}\label{prop:eigen_weak_convergence}
  Under either parameter regime of Section~\ref{sec:prelim}, with high probability there exist $u_i\in\dR^n$, $i\le r_0$, and jointly defined Galton--Watson limits $X^{(a)}\in\dR^{r_0}$, $a\le r$, such that:
  \begin{enumerate}
    \item For the real eigenbasis chosen after Lemma~\ref{lem:cluster_semisimplicity},
    \begin{align}
      \lambda_i&=\mu_i+O(\epsilon_n),\label{eq:eigenvalue_location}\\
      \left\|y_i-\frac{u_i}{\|u_i\|}\right\|&=O(\epsilon_n).\label{eq:approx_y}
    \end{align}
    \item The empirical moment identity below holds simultaneously for the tests
    \[
      f_a(z,s)=z_i\phi_j(a)\quad(i\le r_0,\ j\le r),\qquad
      f_a(z,s)=(s-1)^b z_i z_j\quad(i,j\le r_0,\ b=0,1,2):
    \]
    \begin{align}\label{eq:finite_moment_law}
      \frac1n\sum_{x=1}^n f_{\sigma(x)}(u_1(x),\ldots,u_{r_0}(x),\deg(x))
      =\sum_{a=1}^r\pi_a\E*{f_a(X_1^{(a)},\ldots,X_{r_0}^{(a)},\tilde d)}+O(\epsilon_n).
    \end{align}
    Here $\tilde d$ is the root degree of the Galton--Watson tree defining $X^{(a)}$, and has a Poisson distribution with mean $d$. The limits have means $\mathbb E X_i^{(a)}=\phi_i(a)$ and satisfy
    \[
      (X_i^{(a)}:i\le r_0)\overset{(d)}=
      \left(\mu_i^{-1}\sum_{k=1}^{\tilde d}Y_{ik}:i\le r_0\right),
    \]
    where, conditional on the i.i.d. child types $a_k$ with $\Pr(a_k=b)=Q_{ab}/d$, the vectors $(Y_{ik}:i\le r_0)$ are independent copies of $X^{(a_k)}$.
  \end{enumerate}
\end{proposition}
\begin{proof}
First fix a $\pi$-orthonormal eigenbasis of $Q$, independently of the graph. On a typed Poisson Galton--Watson tree rooted at type $a$, write
\[
 Z_{i,t}=\mu_i^{-t}\sum_{|v|=t}\phi_i(\sigma(v)),\qquad
 m_{ij,t}(a)=\mathbb E_a[Z_{i,t}Z_{j,t}].
\]
By \citep[Propositions~4--5]{stephan2022sparse}, these martingales have means $\phi_i(a)$ and converge jointly in $L^2$ to limits $X_i^{(a)}$. Their limits satisfy the stated branching recursion, and Equation~(8.7) there gives
\begin{equation}\label{eq:finite_depth_moment}
 m_{ij,t}=\frac{Qm_{ij,t-1}}{\mu_i\mu_j}+\phi_i\circ\phi_j
 =\sum_{k=0}^t\left(\frac{Q}{\mu_i\mu_j}\right)^k(\phi_i\circ\phi_j).
\end{equation}
Writing $m_{ij,\infty}(a)=\mathbb E_a[X_i^{(a)}X_j^{(a)}]$, this series gives
\[
 \|m_{ij,t}-m_{ij,\infty}\|_\infty
 \le\frac{\pi_{\min}^{-1}c^{t+1}}{1-c}.
\]
For the degree weights, let $N$ be the root degree. Separating contributions from equal and distinct children gives, for $t\ge1$ and $b=0,1,2$,
\begin{equation}\label{eq:degree_weighted_tree_moment}
 \begin{aligned}
 \mathbb E_a[(N-1)^bZ_{i,t}Z_{j,t}]
 &=\frac{\mathbb E[N(N-1)^b]}{d\mu_i\mu_j}[Qm_{ij,t-1}](a)+\frac{\mathbb E[N(N-1)^{b+1}]}{d^2}\phi_i(a)\phi_j(a).
 \end{aligned}
\end{equation}
The same identity holds for the limits after replacing $m_{ij,t-1}$ by $m_{ij,\infty}$. The Poisson moments then bound the difference between the finite-depth and limiting weighted expectations by
\[
 \frac{(d^2+d)\pi_{\min}^{-1}c^{t+1}}{1-c}.
\]

Set $t=\ell+1$. When the depth-$t$ neighborhood of $x$ contains at most one cycle, define
\[
 u_i(x)=\mu_i^{-t}
 \sum_{\substack{v_0=x,\ldots,v_t\\\text{non-backtracking path in }G}}
 \phi_i(\sigma(v_t)),
\]
and set $u_i(x)=0$ otherwise. This is the path functional in \citep[Equation~(10.1) and Lemma~19]{stephan2022sparse}, with $q=2$ and the indicated normalization. With high probability all these neighborhoods are tangle-free \citep[Lemma~11(i)]{stephan2022sparse}. The tests in (ii) are bounded by $C\pi_{\min}^{-1}(1+|(G,x)_t|)^4$, where $|(G,x)_t|$ counts the neighborhood vertices. Thus \citep[Proposition~7]{stephan2022sparse}, a union bound over the $O(r^2)$ tests, and the preceding tree-moment estimates give \eqref{eq:finite_moment_law}, for this fixed basis, with error at most
\[
 C L_n^C\left[(\log n)^C n^{5\kappa-1/2}+c^{\ell+1}\right]
 =O(\epsilon_n^2).
\]
Here $p_{\max}^t\le L_n n^\kappa$ accounts for the extra depth step. The last estimate holds in both parameter regimes by taking $\kappa<1/20$ and $\alpha<\frac\kappa4\log(1/c)$, since $L_n\ge d\ge c^{-1}$ and $\log(L_n^C(\log n)^C)=o(\log n/\log L_n)$.
This proves (ii) for the fixed basis.

To prove (i), for sufficiently small $\alpha$ in \eqref{eq:def_epsn},  the errors in Lemma~\ref{lem:SZ_precise_inputs}(i) and the approximation in Lemma~\ref{lem:cluster_count_and_control} are $O(\epsilon_n^2)$. For each informative $\mu$, let $\xi_i$, $i\in I_\mu$, be the selected real edge eigenvectors of unit length, and choose real unit vectors $\zeta_i\in W_\mu$ satisfying
\[
 \zeta_i=\frac{\mathsf U_\mu a_i}{\sqrt{\gamma_{i,U}^{(\ell)}}},\qquad
 \|\xi_i-\zeta_i\|=O(\epsilon_n^2),\qquad C_\mu=[a_i]_{i\in I_\mu}.
\]
The two matrix estimates in Lemma~\ref{lem:SZ_precise_inputs}(i), together with $\mu\mathsf U_\mu^\top J\mathsf U_\mu=I+O(\epsilon_n^2)$, give
\begin{align*}
 \|a_i\|^2&=1+O(d\epsilon_n^2), \quad
 a_i^\top a_j
 =\mu\gamma_{i,U}^{(\ell)}\zeta_i^\top J\zeta_j+O(\epsilon_n^2)
   =O_c(d\epsilon_n^2),\qquad i\ne j.
\end{align*}
Here the off-diagonal estimate uses $\xi_i^\top J\xi_j=0$, and $\gamma_{i,U}^{(\ell)}$ is constant on $I_\mu$. Thus, for all large $n$, the orthogonal matrix
\[
 O_\mu=C_\mu(C_\mu^\top C_\mu)^{-1/2}
\]
satisfies
\[
 \|C_\mu^\top C_\mu-I\|+\|C_\mu-O_\mu\|
 =O_c(|I_\mu|d\epsilon_n^2).
\]
For the fixed basis used above, write the moment errors as
\begin{align*}
 F_{ij}&=\frac1n\langle u_i,\phi_j\circ\sigma\rangle-\delta_{ij},
       &&i\le r_0,\ j\le r,\\
 (E_b)_{ij}&=\frac1n\langle u_i,(D-I)^b u_j\rangle
       -\sum_a\pi_a\mathbb E_a[(N-1)^bX_i^{(a)}X_j^{(a)}],
       &&i,j\le r_0.
\end{align*}
The estimates just proved give $\|F\|,\|E_b\|=O(r\epsilon_n^2)=O(\epsilon_n)$ for $b=0,1,2$, since $r\le L_n$ and $L_n\epsilon_n\to0$.

Rotate the $\pi$-orthonormal $\mu$-eigenbasis of $Q$ by $O_\mu$. Since these rotations mix only equal $\mu_i$, the path vectors and tree coordinates undergo the same linear change of basis, and $F$ and $E_b$ are multiplied on both sides by orthogonal matrices. Their operator norms are unchanged, so (ii) holds for this chosen basis even though it depends on $G$.

Since $\|\mathsf U_\mu/\sqrt{\gamma_{i,U}^{(\ell)}}\|=O(1)$, the rotated edge pseudo-eigenvectors, divided by $\sqrt{\gamma_{i,U}^{(\ell)}}$, approximate the $\xi_i$ with the same error.

To obtain vertex vectors, sum the coordinates over directed edges leaving each vertex, as in the proof of Lemma~\ref{lem:cluster_semisimplicity}. On the tangle-free event above, the definitions give
\[
 \sum_{v:\{x,v\}\in E}
 \frac{(B^\ell J\chi_i)_{(x,v)}}{\mu_i^{\ell+1}\sqrt n}
 =\frac{u_i(x)}{\sqrt n}.
\]

We next bound the error introduced by this summation. For any vector $f$ on directed edges, Cauchy--Schwarz gives
\[
 \sum_{x\in V}\left|\sum_{v:\{x,v\}\in E}f_{(x,v)}\right|^2
 \le\bigl(\max_x\deg(x)\bigr)
       \sum_{(x,v)\in\vec E}|f_{(x,v)}|^2.
\]
A Chernoff bound gives $\max_x\deg(x)=O(d+\log n)$ with high probability. Thus the error increases by at most a factor $O(\sqrt{d+\log n})$.

Finally, we check that normalization preserves this error bound. The linear moment identity in (ii) gives $\langle u_i,\tilde\phi_i\rangle/n=1+O(\epsilon_n)$, while $\pi$-orthonormality gives $\|\tilde\phi_i\|=\sqrt n$. Hence
\[
 1-O(\epsilon_n)
 \le\frac{|\langle u_i,\tilde\phi_i\rangle|}{n}
 \le\frac{\|u_i\|\,\|\tilde\phi_i\|}{n}
 =\frac{\|u_i\|}{\sqrt n}.
\]
Since $\gamma_{i,U}^{(\ell)}\le(1-c)^{-1}$, the norm of $u_i/\sqrt{n\gamma_{i,U}^{(\ell)}}$ is therefore bounded below by a positive constant depending only on $c$. Normalizing the two vertex vectors now gives
\[
 \left\|y_i-\frac{u_i}{\|u_i\|}\right\|
 =O_c\!\left(|I_\mu|d\sqrt{d+\log n}\,\epsilon_n^2\right)
 =O(\epsilon_n).
\]
The last equality follows from $|I_\mu|d\le L_n^2$ and $L_n^2\sqrt{L_n+\log n}\,\epsilon_n\to0$. The eigenvalue estimate follows from Lemma~\ref{lem:SZ_precise_inputs}(ii), proving (i).
\end{proof}

Applying Proposition~\ref{prop:eigen_weak_convergence}(ii) to the linear test $f_a(z,s)=z_i\phi_j(a)$ gives, with high probability, simultaneously for $i\le r_0$ and $j\le r$,
\begin{equation}\label{eq:scal_uphi}
  \begin{aligned}
    \frac1n\langle u_i, \phi_j\circ\sigma \rangle
      &= \sum_{a=1}^r\pi_a\phi_j(a)\mathbb E X_i^{(a)}+O(\epsilon_n)\\
      &= \sum_{a=1}^r\pi_a\phi_i(a)\phi_j(a)+O(\epsilon_n)
       =\delta_{ij}+O(\epsilon_n).
  \end{aligned}
\end{equation}
Here we used $\mathbb E X_i^{(a)}=\phi_i(a)$ and the $\pi$-orthonormality of the eigenvectors of $Q$.
In particular, the above proposition implies the following asymptotic formulas.
\begin{proposition}\label{prop:scal_u}
  For any $i, j \in [r_0]$, the following holds with high probability:
  \begin{align}
    \langle u_i, u_j \rangle &= n(1 - \tau_i)^{-1}\delta_{ij} + O(n\epsilon_n) \label{eq:scal_uu} \\
    \langle u_i, (D-I) u_j \rangle &= n (d(1 - \tau_i)^{-1} + 1) \delta_{ij} + O(n\epsilon_n) \label{eq:scal_udu} \\
    \langle u_i, (D-I)^2 u_j \rangle &= n ((d^2+d)(1 - \tau_i)^{-1} + 2d+1) \delta_{ij} + O(n\epsilon_n) \label{eq:scal_uddu}
  \end{align}
\end{proposition}

\begin{proof}
Let $m_{ij}(a)=\mathbb E_a[X_i^{(a)}X_j^{(a)}]$. Taking the limit in \eqref{eq:finite_depth_moment} gives
\begin{equation}\label{eq:mij_fixedpoint}
 m_{ij}=\frac{Qm_{ij}}{\mu_i\mu_j}+\phi_i\circ\phi_j.
\end{equation}
Since $\pi^\top Q=d\pi^\top$ and $\langle\phi_i,\phi_j\rangle_\pi=\delta_{ij}$,
\[
 \left(1-\frac d{\mu_i\mu_j}\right)\langle\pi,m_{ij}\rangle
 =\delta_{ij},\qquad
 \langle\pi,m_{ij}\rangle=(1-\tau_i)^{-1}\delta_{ij}.
\]
Let $N$ be the root degree of the same tree. Averaging the limiting identity \eqref{eq:degree_weighted_tree_moment} against $\pi$, and using the Poisson moments and \eqref{eq:mij_fixedpoint}, gives
\begin{align*}
 \sum_a\pi_a\mathbb E_a[(N-1)X_i^{(a)}X_j^{(a)}]
   &=d\langle\pi,m_{ij}\rangle+\delta_{ij},\\
 \sum_a\pi_a\mathbb E_a[(N-1)^2X_i^{(a)}X_j^{(a)}]
   &=(d^2+d)\langle\pi,m_{ij}\rangle+(2d+1)\delta_{ij}.
\end{align*}
Proposition~\ref{prop:eigen_weak_convergence}(ii) now gives, for $b=0,1,2$,
\[
 \langle u_i,(D-I)^b u_j\rangle
 =n\sum_a\pi_a\mathbb E_a[(N-1)^bX_i^{(a)}X_j^{(a)}]
       +O(n\epsilon_n).
\]
Substituting the three averaged moments proves \eqref{eq:scal_uu}--\eqref{eq:scal_uddu}.
\end{proof}

\begin{corollary}\label{cor:inner_product}
  For any $i, j \in [r_0]$, the following holds with high probability:
  \begin{align}
    \langle y_i, y_j \rangle &= \delta_{ij} + O(\epsilon_n), \label{eq:scal_yy} \\
    \langle x_i, y_j \rangle &= \frac{d+1-\tau_i}{\mu_i} \delta_{ij} +O(\epsilon_n), \label{eq:scal_xy} \\
    \langle x_i, x_j \rangle &= \frac{d^2 + d + (2d+1)(1-\tau_i)}{\mu_i^2}\delta_{ij} +O(\epsilon_n) ,\label{eq:innerxx}
  \end{align}
  As a result,
  \begin{equation}\label{eq:projxy}
    \norm*{x_i - \langle x_i, y_i \rangle y_i}^2 = \frac{d + \tau_i(1 - \tau_i)}{\mu_i^2} + O(\epsilon_n).
  \end{equation}

\end{corollary}

\begin{proof}
Decrease $\alpha$ in \eqref{eq:def_epsn}, if necessary, so that \eqref{eq:eigenvalue_location}, \eqref{eq:approx_y}, and \eqref{eq:scal_uu}--\eqref{eq:scal_uddu} hold with $\epsilon_n^2$ in place of $\epsilon_n$. Since $1-\tau_i\ge1-c>0$, these estimates give $\|u_i\|=\sqrt{n/(1-\tau_i)}\,(1+O(\epsilon_n^2))$, and hence
\begin{equation}\label{eq:yapproxu}
 \left\|y_i-\frac{u_i}{\sqrt{n/(1-\tau_i)}}\right\|=O(\epsilon_n^2).
\end{equation}
The vertex-degree bound used in Proposition~\ref{prop:eigen_weak_convergence} gives $\|D-I\|=O(d+\log n)$ with high probability. Substituting \eqref{eq:yapproxu} directly into the weighted inner products gives, for $b=0,1,2$,
\[
 \langle y_i,(D-I)^b y_j\rangle
 =\frac{\sqrt{(1-\tau_i)(1-\tau_j)}}n
       \langle u_i,(D-I)^b u_j\rangle
   +O((d+\log n)^b\epsilon_n^2).
\]
Moreover, $(D-I)y_i=\lambda_i x_i$ implies
\[
 \langle x_i,y_j\rangle
    =\frac{\langle y_i,(D-I)y_j\rangle}{\lambda_i},\qquad
 \langle x_i,x_j\rangle
    =\frac{\langle y_i,(D-I)^2y_j\rangle}{\lambda_i\lambda_j}.
\]
Applying Proposition~\ref{prop:scal_u} and $\lambda_i=\mu_i+O(\epsilon_n^2)$ proves the first three formulas, with errors bounded by $O((L_n+\log n)^2\epsilon_n^2)$.

For the final formula, keep $\lambda_i$ in the denominator until after the cancellation. Since $\|y_i\|=1$,
\begin{align*}
 \|x_i-\langle x_i,y_i\rangle y_i\|^2
 &=\frac{\langle y_i,(D-I)^2y_i\rangle
                  -\langle y_i,(D-I)y_i\rangle^2}{\lambda_i^2}\\
 &=\frac{d^2+d+(2d+1)(1-\tau_i)-(d+1-\tau_i)^2}{\mu_i^2}
       +O((L_n+\log n)^2\epsilon_n^2)\\
 &=\frac{d+\tau_i(1-\tau_i)}{\mu_i^2}
       +O((L_n+\log n)^2\epsilon_n^2).
\end{align*}
In either parameter regime,
\[
 (L_n+\log n)^2\epsilon_n
 =\exp\!\left(2\log(L_n+\log n)
       -\alpha\frac{\log n}{\log L_n}\right)\longrightarrow0.
\]
Thus all the errors above are $O(\epsilon_n)$.
\end{proof}

Next, we  show that $y_1,\dots,y_{r_0}$ span an $r_0$-dimensional subspace.
\begin{lemma}\label{lem:linearindependence}
  With high probability, $y_1,\dots,y_{r_0}$ are linearly independent.
\end{lemma}

\begin{proof}
Let $K=(\langle y_i,y_j\rangle)_{i,j\le r_0}$. By \eqref{eq:scal_yy},
\[
 \|K-I\|\le r_0\max_{i,j\le r_0}|K_{ij}-\delta_{ij}|
       =O(r_0\epsilon_n)=o(1),
\]
since $r_0\le L_n$ and $L_n\epsilon_n\to0$. Thus $K$ is positive definite, and $y_1,\ldots,y_{r_0}$ are linearly independent.
\end{proof}


The next lemma provides a way to relate $H(t)$ and $H(t')$ for two different $t, t'$:

\begin{lemma}\label{lem:Hy}
  Let $t, t' \neq 0$. Then
  \[ \frac{H(t)}{t} - \frac{H(t')}{t'} = (t - t')\left(I- \frac{D-I}{tt'}\right).\]
  As a result, for any $t \in \dR$,
  \[ H(t)y_i = (ty_i - x_i)(t - \lambda_i).\]
\end{lemma}

\begin{proof}
The first identity follows from $H(t)/t=tI-A+(D-I)/t$. For the second, the eigenvector equations \eqref{eq:BHeq}--\eqref{eq:Bheq1} give $Ay_i=x_i+\lambda_i y_i$ and $(D-I)y_i=\lambda_i x_i$. Thus, for every $t\in\mathbb R$,
\[
 H(t)y_i=t^2y_i-t(x_i+\lambda_i y_i)+\lambda_i x_i
        =(t-\lambda_i)(ty_i-x_i).
\]
\end{proof}

\section{Negative eigenvalues of \texorpdfstring{$H(t)$}{H(t)}}\label{sec:negative}

The proof of Theorem \ref{thm:count} is based on showing both an upper and a lower bound on the number of negative eigenvalues. The upper bound is shown using a strengthening of the method in \citep{mohanty2024robust}, while the lower bound is based on the Courant-Fischer min-max principle.

\subsection{Upper bound on the number of negative eigenvalues}
We provide a stronger lemma compared to \citep[Lemma 5.4]{mohanty2024robust} for the deterministic relation between negative eigenvalues of $H(t)$ and real eigenvalues of $\tilde B$. We also provide a quantitative relation between the spectrum of $H(\pm\sqrt{d})$ and $B$.
\begin{lemma}\label{lem:lower_bound}
  The negative eigenvalues of $H(t)$ and the real eigenvalues of $\tilde B$ satisfy the following relations:
  \begin{enumerate}
    \item   For any $t>0$, the number of negative eigenvalues of $H(t)$ is at most the number of real eigenvalues of $\tilde B$ larger than $t$.
    \item For any $t<0$, the number of negative eigenvalues of $H(t)$ is at most the number of real eigenvalues of $\tilde B$ smaller than $t$.
    \item For $t=0$, the number of negative eigenvalues of $H(0)$ is at most the multiplicity  of the  eigenvalue $-1$ of $\tilde B$.
    \item For $d>1$, $\delta>0$, and $s\in\{-1,1\}$,
    \[
      N_{2\sqrt d\,\delta}(s\sqrt d)\le N(s(\sqrt d+\delta)).
    \]
    By (i) and (ii), the right-hand side is at most the number of real eigenvalues of $\tilde B$ larger than $\sqrt d+\delta$ if $s=1$, or smaller than $-\sqrt d-\delta$ if $s=-1$.
  \end{enumerate}
\end{lemma}

\paragraph{Proof of (i)} For $t\not=0$, we can write
\[ H(t)=t^2\left(I-\frac{1}{t} A+\frac{1}{t^2} (D-I)\right).\]
Let $k_-$ be the number of negative eigenvalues of $H(t_*)$ for $t_*>0$.
Since eigenvalues of $H(t)$ are continuous in $t$,  for sufficiently large $t>t_*$, $H(t)$ is positive definite. Therefore we have
\begin{align}\label{eq:continuity}
  k_-\leq  \sum_{t>t_*, \det H(t)=0} \dim \ker H(t).
\end{align}
This is because all negative eigenvalues of $H(t_*)$ will cross $0$ at least once as  $t$ increases.   From \eqref{eq:Ihara}, when $H(t)$ is singular, $t$ is an eigenvalue of $\tilde B$. Let $k=\dim \ker H(t)$. To compare multiplicities, set $h(s)=\det H(s)=\det(\tilde B-sI)$ for $s\neq0$. For every $a<k$, the derivative expansion of \citep[Lemma 3.3]{mohanty2024robust} gives
\begin{align}
  h^{(a)}(t)=\sum_{S,T\subset [n], |S|=|T|\geq n-a} \det(H(t)_{S,T}) q^{(a)}_{S,T}(t),
\end{align}
where $H(t)_{S,T}$ is the submatrix of $H(t)$ on indices $S,T$, and $q^{(a)}_{S,T}(t)$ are polynomials in $t$. Since $\rank H(t)=n-k$, every minor of size strictly larger than $n-k$ is zero. For $a<k$ we have $n-a>n-k$, hence $h^{(a)}(t)=0$. Thus $t$ is a root of $\det(\tilde B-sI)$ of order at least $k$. This finishes the proof for Case (i).
~
\paragraph{Proof of (ii)} For $t_*<0$, $H(t)$ is positive definite for sufficiently small $t<t_*$. By the same continuity argument,
\begin{align}
  k_-\leq  \sum_{t<t_*, \det H(t)=0} \dim \ker H(t).
\end{align}
Repeating the proof implies Case (ii).

\paragraph{Proof of (iii)} When $t=0$, $H(0)=D-I$. Then, the number of negative eigenvalues of $H(0)$ is the number of vertices with degree $0$. Let $s_1,\dots,s_k\in [n]$ be the indices of these vertices of degree $0$. Since each $s_j$ is isolated, $(D-I)\mathbf 1_{s_j}=-\mathbf 1_{s_j}$ and $A\mathbf 1_{s_j}=0$. Hence
\[
\begin{bmatrix}
  \mathbf 1_{s_j}\\
  \mathbf 1_{s_j}
\end{bmatrix}\in \dR^{2n}
\]
satisfies \eqref{eq:BHeq} and \eqref{eq:Bheq1} with $\lambda=-1$. These $k$ vectors have disjoint supports and are linearly independent, so the geometric, and therefore algebraic, multiplicity of $-1$ is at least $k$. This concludes the proof.

\paragraph{Proof of (iv)} For  $s=\pm 1$, a direct calculation gives
\begin{align}\label{eq:one_sided_comparison}
 H(s(\sqrt d+\delta))
 &=\left(1+\frac\delta{\sqrt d}\right)H(s\sqrt d)
    +\delta\left[(\sqrt d+\delta)I-\frac{D-I}{\sqrt d}\right]\notag\\
 &\preceq\left(1+\frac\delta{\sqrt d}\right)H(s\sqrt d)
    +\delta\left(\sqrt d+\delta+\frac1{\sqrt d}\right)I,
\end{align}
since $D-I\succeq-I$. For every nonzero vector $y$ in the subspace spanned by eigenvectors of $H(s\sqrt d)$ with eigenvalues below $-2\sqrt d\,\delta$, this gives
\[
 y^\top H(s(\sqrt d+\delta))y
 <-\delta\left(\sqrt d-\frac1{\sqrt d}+\delta\right)\|y\|^2<0.
\]
The Courant-Fischer principle therefore gives $N_{2\sqrt d\,\delta}(s\sqrt d)\le N(s(\sqrt d+\delta))$.

\subsection{Quadratic forms on the vectors \texorpdfstring{$y_i$}{y i}}
The following estimate will give the negative subspaces used in the proof of Theorem~\ref{thm:count}.
  We define $\lambda_k^+$ (resp. $\lambda_k^-$) as the eigenvalue of $\tilde B$ with the $k$-th largest (resp. $k$-th lowest) real part. Its associated eigenvector is denoted $\dbinom{x^+_k}{y^+_k}$ and $\dbinom{x^-_k}{y^-_k}$, respectively. From Proposition \ref{prop:eigen_weak_convergence}, we have with high probability,
  \[\lambda_k^+ = \mu_k^+ + O(\epsilon_n) \quad \text{for } k\in [r_+], \]
  and the same holds for the $\lambda_k^-$.

\begin{lemma}[Quadratic-form estimate]\label{lem:eigen_BH}
Use the real eigenvectors chosen in Proposition~\ref{prop:eigen_weak_convergence}. Under either parameter regime of Section~\ref{sec:prelim}, for sufficiently small $\alpha$ in \eqref{eq:def_epsn}, the following holds with high probability, simultaneously for all $i,j\le r_0$ and $t\in\mathbb R$:
\begin{equation}\label{eq:yHy}
 y_j^\top H(t)y_i
 =\left(t-\frac{d+1-\tau_i}{\mu_i}\right)(t-\mu_i)\delta_{ij}
   +O((1+|t|)(|t|+d)\epsilon_n^2).
\end{equation}
\end{lemma}
\begin{proof}
Choose $\alpha$ sufficiently small that \eqref{eq:eigenvalue_location}, \eqref{eq:scal_yy}, and \eqref{eq:scal_xy} give, simultaneously for all $i,j\le r_0$,
\begin{align*}
 \lambda_i&=\mu_i+O(\epsilon_n^2),\\
 \langle y_i,y_j\rangle&=\delta_{ij}+O(\epsilon_n^2),\\
 \langle x_i,y_j\rangle
 &=\frac{d+1-\tau_i}{\mu_i}\delta_{ij}+O(\epsilon_n^2).
\end{align*}
Lemma~\ref{lem:Hy} then yields
\begin{align*}
 y_j^\top H(t)y_i
 &=(t\langle y_i,y_j\rangle-\langle x_i,y_j\rangle)(t-\lambda_i)\\
 &=\left[\left(t-\frac{d+1-\tau_i}{\mu_i}\right)\delta_{ij}
          +O((1+|t|)\epsilon_n^2)\right]
     \left(t-\mu_i+O(\epsilon_n^2)\right).
\end{align*}
Since $|\mu_i|\le d$ and
$|(d+1-\tau_i)/\mu_i|\le(d+1)/\sqrt d\le2\sqrt d$,
expanding the product proves \eqref{eq:yHy}. \end{proof}

  \subsection{Proof of Theorem~\ref{thm:count}}\label{sec:proof_count}
We work on the common high-probability event where Lemma~\ref{lem:cluster_semisimplicity} gives real informative outliers of $\tilde B$, and Corollary~\ref{cor:inner_product} and Lemma~\ref{lem:eigen_BH} give \eqref{eq:scal_yy} and \eqref{eq:yHy}, respectively. On the event in Lemma~\ref{lem:SZ_precise_inputs}(ii), the determinant identity \eqref{eq:det_tildB} also gives the bound $\sqrt d+O(\log\log n/\log n)$ for the modulus of every remaining eigenvalue of $\tilde B$, since $d$ and $L_n$ are fixed.

We first construct a subspace on which the quadratic form is uniformly negative for parameters in a fixed neighborhood of $\sqrt d$. For $m>\sqrt d$ and $d\ge2$,
\begin{equation}\label{eq:numerical1}
 a(m):=\frac{d+1-d/m^2}{m}<\sqrt d.
\end{equation}
Indeed, $a(\sqrt d)=\sqrt d$ and $a'(m)=-(d+1)/m^2+3d/m^4<0$ for $m>\sqrt d$ when $d\ge2$. For each informative eigenvalue, write $m_i=|\mu_i|$ and $a_i=a(m_i)$. Then $a_i<\sqrt d<m_i$, and hence $(\sqrt d-a_i)(\sqrt d-m_i)<0$. Since there are finitely many informative values and they are independent of $n$, we can choose fixed $\eta_0,\gamma>0$ such that
\[
 (t-a_i)(t-m_i)\le-2\gamma\quad\text{whenever }|t-\sqrt d|\le\eta_0
\]
for every informative index. Let $E^\pm=\operatorname{span}\{y_i^\pm:i\in[r_\pm]\}$. Equations \eqref{eq:scal_yy} and \eqref{eq:yHy}, with matching signs, therefore give
\begin{equation}\label{eq:uniform_trial_negative}
 y^\top H(\pm t)y\le-\gamma\|y\|^2,\qquad y\in E^\pm,\quad |t-\sqrt d|\le\eta_0,
\end{equation}
for all sufficiently large $n$. For either sign, this lower bound is vacuous if the corresponding $r_\pm$ is zero.

Choose $\varepsilon_0<\min\{\eta_0/3,\gamma/2\}$. The following argument applies both to $\varepsilon_n=\varepsilon$ for a fixed $0<\varepsilon\le\varepsilon_0$, and to $\varepsilon_n=\log^{-a}n$ for a fixed $a\in(0,1)$. In both cases, $0<\varepsilon_n\le\varepsilon_0$ for all sufficiently large $n$ and $\log\log n/\log n=o(\varepsilon_n)$.  Lemma~\ref{lem:lower_bound}(i)--(ii) gives
$N(\pm(\sqrt d+\varepsilon_n))\le r_\pm$.

For the lower bounds, Lemma~\ref{lem:linearindependence} gives $\dim E^\pm=r_\pm$. Since $\varepsilon_n\le\varepsilon_0<\eta_0$, we may apply \eqref{eq:uniform_trial_negative} with $t=\sqrt d+\varepsilon_n$ to obtain
\[
 y^\top H\bigl(\pm(\sqrt d+\varepsilon_n)\bigr)y
 \le-\gamma\|y\|^2<0,\qquad 0\ne y\in E^\pm.
\]
The Courant--Fischer principle therefore gives
\[
 N(\sqrt d+\varepsilon_n)\ge r_+,
 \qquad N(-\sqrt d-\varepsilon_n)\ge r_-.
\]

Since $d$ is fixed and $\log\log n/\log n=o(\varepsilon_n)$, every eigenvalue of $\tilde B$, other than its $r_0$ informative outliers, has modulus strictly smaller than $\sqrt d+\varepsilon_n/(2\sqrt d)$ for all sufficiently large $n$. Applying Lemma~\ref{lem:lower_bound}(iv) with $\delta=\varepsilon_n/(2\sqrt d)$, and then parts (i)--(ii), gives
\begin{align*}
 N_{\varepsilon_n}(\sqrt d)
 &\le N\left(\sqrt d+\frac{\varepsilon_n}{2\sqrt d}\right)\le r_+,\\
 N_{\varepsilon_n}(-\sqrt d)
 &\le N\left(-\sqrt d-\frac{\varepsilon_n}{2\sqrt d}\right)\le r_-.
\end{align*}
At $t=\sqrt d$, \eqref{eq:uniform_trial_negative} and $\varepsilon_n\le\varepsilon_0<\gamma/2$ give
\[
 y^\top H(\pm\sqrt d)y\le-\gamma\|y\|^2
 <-\varepsilon_n\|y\|^2,\qquad 0\ne y\in E^\pm.
\]
Since $\dim E^\pm=r_\pm$, the Courant--Fischer principle gives
\[
 N_{\varepsilon_n}(\sqrt d)\ge r_+,
 \qquad N_{\varepsilon_n}(-\sqrt d)\ge r_-.
\]
This proves both identities in Theorem~\ref{thm:count}.

  \section{Eigenvectors in the higher-degree regime}\label{sec:high}

We first compare the action of $H(\pm\sqrt d)$ on the orthonormalized vertex vectors with the diagonal matrix of the proposed eigenvalue locations.

\begin{lemma}\label{lem:matrix_eigenvector_bound}
Write $Y^+=[y_i^+]_{i\le r_+}$, $\widetilde Y^+=Y^+((Y^+)^\top Y^+)^{-1/2}$, and
\[
 {\Lambda^+}=\operatorname{diag}\left((\sqrt d-\mu_i^+)\left(\sqrt d-\frac d{\mu_i^+}\right)\right).
\]
In either parameter regime of Section~\ref{sec:prelim}, for all sufficiently large $n$,
\[
 \|H(\sqrt d)\widetilde Y^+-\widetilde Y^+ {\Lambda^+}\|\le\sqrt d+2
\]
with high probability. The same statement holds with all $+$ superscripts and subscripts replaced by $-$, $H(\sqrt d)$ replaced by $H(-\sqrt d)$, and ${\Lambda^-}=\operatorname{diag}((-\sqrt d-\mu_i^-)(-\sqrt d-d/\mu_i^-))$.
\end{lemma}
\begin{proof}
We prove the $+$ case; the same argument gives the $-$ case. Assume $r_+>0$. We will use $rd^3\epsilon_n\to0$. For fixed parameters this follows from \eqref{eq:def_epsn}; under Assumption~\ref{assumption:uniform},
\[
 rd^3\epsilon_n\le(\log n)^{4K}
 \exp\left(-\frac{\alpha\log n}{K\log\log n}\right)\longrightarrow0.
\]
Define the matrix
\[
 W=\left[(\sqrt d-\mu_i^+)
       \left(x_i^+-\frac d{\mu_i^+}y_i^+\right)\right]_{i\le r_+}.
\]
By Lemma~\ref{lem:Hy}, the $i$-th column of $H(\sqrt d)Y^+-Y^+ {\Lambda^+}+W$ is exactly $(\mu_i^+-\lambda_i^+)(\sqrt d y_i^+-x_i^+)$. Equation~\eqref{eq:innerxx} gives $\|x_i^+\|=O(\sqrt d)$, so \eqref{eq:eigenvalue_location} implies
\[
 \|H(\sqrt d)Y^+-Y^+ {\Lambda^+}+W\|
 =O(\sqrt{r_+ d}\,\epsilon_n)=o(1).
\]
The three inner product formulas in Corollary~\ref{cor:inner_product} give the cancellation
\[
 \left\langle x_i^+-\frac d{\mu_i^+}y_i^+,
                 x_j^+-\frac d{\mu_j^+}y_j^+\right\rangle
 =\delta_{ij}\frac{d+1-\tau_i^+}{(\mu_i^+)^2}
   +O(d\epsilon_n).
\]
Since $|\sqrt d-\mu_i^+|\le2d$, passing from these entrywise errors to the operator norm yields
\[
 W^\top W=
 \operatorname{diag}\left((1-\sqrt{\tau_i^+})^2
                         (d+1-\tau_i^+)\right)
 +O(r_+ d^3\epsilon_n).
\]
Every diagonal entry is at most $d+1$. Consequently,
\[
 \|H(\sqrt d)Y^+-Y^+ {\Lambda^+}\|\le\sqrt{d+1}+o(1).
\]

It remains to replace $Y^+$ by $\widetilde Y^+$. Put $K^+=(Y^+)^\top Y^+$. By \eqref{eq:scal_yy},
\[
 \|(K^+)^{-1/2}-I\|=O(r_+\epsilon_n),
 \qquad \|Y^+\|=1+O(r_+\epsilon_n).
\]
Using $\widetilde Y^+=Y^+(K^+)^{-1/2}$ and $\|{\Lambda^+}\|=O(d^2)$, we obtain
\begin{align*}
 \|H(\sqrt d)\widetilde Y^+-\widetilde Y^+ {\Lambda^+}\|
 &\le\|H(\sqrt d)Y^+-Y^+ {\Lambda^+}\|\,\|(K^+)^{-1/2}\|+2\|Y^+\|\,\|{\Lambda^+}\|\,\|(K^+)^{-1/2}-I\|\\
 &\le\sqrt{d+1}+o(1)\le\sqrt d+2
\end{align*}
for all sufficiently large $n$. Here the additional terms are $O(r_+ d^2\epsilon_n)=o(1)$ by the first bound of the proof.
\end{proof}

\begin{proposition}\label{prop:eigendecomposition_H}
Assume Assumption~\ref{assumption:uniform} and $\tau_{r_0}\le c<1$ for a fixed $c\in(0,1)$. For $d\ge d_c$, the following holds with high probability. For either choice of sign, use the ordering in \eqref{eq:mu_order} and set
\[
 a_k^\pm=(\pm\sqrt d-\mu_k^\pm)\left(\pm\sqrt d-\frac d{\mu_k^\pm}\right).
\]
Then
\begin{enumerate}
 \item For $k\le r_\pm$,
 \[
  \bigl|\lambda_{n-k+1}(H(\pm\sqrt d))-a_k^\pm\bigr|
     \le2(\sqrt d+2).
 \]
 \item If $V^\pm$ is an orthonormal eigenbasis for the $r_\pm$ smallest eigenvalues and $Y^\pm=[y_k^\pm]$, there is an orthogonal $O^\pm$ such that
 \[
  \|V^\pm-Y^\pm O^\pm\|_F
   \le\frac{C\sqrt{r_\pm}(\sqrt d+2)}{d (1-c)^2}
 \]
 for a universal constant $C$.
\end{enumerate}
\end{proposition}
\begin{proof}
The conclusions are vacuous when $r_\pm=0$, so suppose the chosen sign satisfies $r_\pm\ge1$. Put $M=H(\pm\sqrt d)$, $Q^\pm=\widetilde Y^\pm$, $P^\pm=Q^\pm(Q^\pm)^\top$, and ${\Lambda^\pm}=\operatorname{diag}(a_k^\pm)$. Lemma~\ref{lem:matrix_eigenvector_bound} and the block comparison in Lemma~\ref{lem:app:local_weyl} give
\begin{align*}
 S&=Q^\pm {\Lambda^\pm}(Q^\pm)^\top+(I-P^\pm)M(I-P^\pm),\quad
 \|M-S\|\le2(\sqrt d+2).
\end{align*}
By Lemma~\ref{lem:SZ_precise_inputs}(ii), choosing
\[
 \delta=C\sqrt d\,\frac{\log L_n\,\log\log n}{\log n}
\]
with $C$ sufficiently large puts every eigenvalue of $\tilde B$, other than its $r_0$ informative outliers, inside the disk of radius $\sqrt d+\delta$. For sufficiently large $n$, the outliers associated with positive $\mu_i$ are positive, and those associated with negative $\mu_i$ are negative. Thus, counting with algebraic multiplicity, $\tilde B$ has at most $r_+$ real eigenvalues satisfying $\lambda>\sqrt d+\delta$, and at most $r_-$ real eigenvalues satisfying $\lambda<-\sqrt d-\delta$. Parts (i), (ii), and (iv) of Lemma~\ref{lem:lower_bound} give
\begin{align*}
 N_{2\sqrt d\,\delta}(\sqrt d)
 &\le N(\sqrt d+\delta)\le r_+,\\
 N_{2\sqrt d\,\delta}(-\sqrt d)
 &\le N(-\sqrt d-\delta)\le r_-.
\end{align*}
Since $M=H(\pm\sqrt d)$, these inequalities imply
\[
 \lambda_{n-r_\pm}(M)\ge-2\sqrt d\,\delta
 =-2Cd\,\frac{\log L_n\,\log\log n}{\log n}.
\]
Since $\log L_n\le K\log\log n$, this error is $o(d)$. For all sufficiently large $n$, the preceding norm bound and Weyl's inequality give
\[
 \lambda_{n-r_\pm}(M)\ge-\frac{d(1-c)^2}{16},
 \qquad \lambda_{n-r_\pm}(S)\ge-\frac{d(1-c)^2}{16}-2(\sqrt d+2).
\]
On the other hand, every signal center satisfies
\begin{equation}\label{eq:signal_center_gap}
 \begin{aligned}
 a_k^\pm&=-d\frac{(1-\sqrt{\tau_k^\pm})^2}{\sqrt{\tau_k^\pm}}=-d\frac{(1-\tau_k^\pm)^2}
          {\sqrt{\tau_k^\pm}(1+\sqrt{\tau_k^\pm})^2}
 \le-\frac{d (1-c)^2}{4}.
 \end{aligned}
\end{equation}
Choose $d_c$, depending only on $c$, sufficiently large that $2(\sqrt d+2)\le d(1-c)^2/16$ for every $d\ge d_c$. Then $\lambda_{n-r_\pm}(S)\ge-d(1-c)^2/8$, while all $r_\pm$ centers are at most $-d(1-c)^2/4$. Since $SQ^\pm=Q^\pm {\Lambda^\pm}$, these centers are its $r_\pm$ smallest eigenvalues: any additional eigenvalue below $-d(1-c)^2/8$ would contradict the bound on $\lambda_{n-r_\pm}(S)$. Their order agrees with \eqref{eq:mu_order}, since $2d-\sqrt d(m+d/m)$ is strictly decreasing for $m>\sqrt d$. Weyl's inequality proves (i).

Let $E^\pm$ be the span of the $r_\pm$ smallest eigenvectors of $M$, and let $P_{E^\pm}$ be its orthogonal projection. The spectrum on $(E^\pm)^\perp$ is bounded below by $-d(1-c)^2/16$. Applying Lemma~\ref{lem:davis_kahan} to each column $q_k$ of $Q^\pm$ gives
\[
 \|(I-P_{E^\pm})q_k\|
 \le\frac{\sqrt d+2}{|a_k^\pm|-d(1-c)^2/16}
 \le\frac{C(\sqrt d+2)}{d (1-c)^2}.
\]
Summing the squared column bounds and using the last conclusion of Lemma~\ref{lem:davis_kahan} gives an orthogonal matrix $R$ such that
\[
 \|V^\pm-Q^\pm R\|_F
 \le\sqrt2\|(I-P_{E^\pm})Q^\pm\|_F
 \le\frac{C\sqrt{r_\pm}(\sqrt d+2)}{d(1-c)^2}.
\]
The normalization estimate in the proof of Lemma~\ref{lem:matrix_eigenvector_bound} also gives
\[
 \|Q^\pm-Y^\pm\|_F
 \le\sqrt{r_\pm}\,\|Q^\pm-Y^\pm\|
 =O(r_\pm^{3/2}\epsilon_n)
 =o\!\left(\sqrt{r_\pm/d}\right).
\]
Indeed, $r_\pm\sqrt d\,\epsilon_n\le rd^3\epsilon_n\to0$. The triangle inequality proves (ii).
\end{proof}

\section{Weak recovery with the Bethe-Hessian}\label{sec:weak_recovery}

We now move on to the proof of Theorem \ref{thm:weak_recovery}. We first show a quantitative version of Corollary \ref{cor:bethe_overlaps}:
\begin{lemma}\label{lem:bethe_overlaps_sharp}
  Assume Assumption~\ref{assumption:uniform}, $d\ge d_c$ for a sufficiently large constant $d_c>0$ depending only on a fixed $c\in(0,1)$, and
    $\max_{i \in [r_0]} \tau_i \leq c$.
  Define the two matrices $\check\Phi^+ \in \dR^{n \times r_+}$ and $\check\Phi^- \in \dR^{n \times r_-}$ by
  \[
    \check\Phi^+_{xk}=n^{-1/2}\phi_k^+(\sigma(x)) \quand
    \check\Phi^-_{xk}=n^{-1/2}\phi_k^-(\sigma(x)),
  \]
  where $\phi_k^\pm$ is a $\pi$-orthonormal eigenvector of $Q$ associated with $\mu_k^\pm$.
  Then there exists an orthogonal matrix $\tilde O^+\in \dR^{r_+ \times r_+}$ and a universal constant $C>0$ such that
  \[
    \norm{V^+ - \check\Phi^+ \tilde O^+}_F
    \leq 2\left(\sum_{i=1}^{r_+}\tau_i^+\right)^{1/2}
    + \frac{C\sqrt{r_+}(\sqrt{d}+1)}{d (1-c)^2},
  \]
  and the same holds for $V^-, \check\Phi^-$.
\end{lemma}

\begin{proof}
The basis choice in Proposition~\ref{prop:eigen_weak_convergence} gives an orthogonal matrix $W^+$, block diagonal with respect to the distinct positive informative eigenvalues of $Q$, such that the columns of $\check\Phi^+W^+$ are matched with the columns of $Y^+$. Equations~\eqref{eq:approx_y}, \eqref{eq:scal_uphi}, and \eqref{eq:scal_uu}, together with the unit norms of the matched columns, give
\begin{align*}
 \|Y^+-\check\Phi^+W^+\|_F^2
 &=2\sum_{i=1}^{r_+}\left(1-\sqrt{1-\tau_i^+}\right)
      +O(r_+\epsilon_n)\\
 &\le4\sum_{i=1}^{r_+}\tau_i^+.
\end{align*}
The last inequality uses $1-\sqrt{1-\tau}\le\tau$, $\tau_i^+\ge1/d$, and $d\epsilon_n\to0$. Proposition~\ref{prop:eigendecomposition_H} gives an orthogonal matrix $O_1^+$ such that
\[
 \|V^+-Y^+O_1^+\|_F
 \le\frac{C\sqrt{r_+}(\sqrt d+1)}{d(1-c)^2}.
\]
The triangle inequality, with $\tilde O^+=W^+O_1^+$, proves the result. The same argument applies to the $-$ case.
\end{proof}

We are now finally able to show Theorem \ref{thm:weak_recovery}. We use the following analysis of approximate $k$-means from \citep{lei2015consistency}:
\begin{lemma}[\citep{lei2015consistency}, Lemma 5.3]\label{lem:k-means-consistency}
  Let $\eta>0$, and let $U,\hat U\in\dR^{n\times r}$ satisfy $U=\Sigma P$ with $\Sigma\in\mathbb M_{n,r}$ and $n\pi_k=\sum_{x=1}^n\Sigma_{xk}>0$ for every $k\in[r]$. Let $(\hat\Sigma,\hat P)$ be a $(1+\eta)$-approximate solution to the $k$-means problem run on $\hat U$. For $k \in [r]$, choose $0<\delta_k\le\min_{\ell\ne k}\norm{P_{k:}-P_{\ell:}}$. Then there exist sets $S_1, \dots, S_r \subseteq [n]$ such that all vertices in $S_k$ have label $k$, and
  \[ \sum_{k=1}^r \delta_k^2 |S_k| \leq (16 + 8\eta) \norm{\hat U - U}_F^2. \]
  Further, if
  \begin{equation}
    (16+8\eta)\norm{\hat U-U}_F^2<n\pi_k\delta_k^2\quad\text{for every }k,\label{eq:k-means-condition}
  \end{equation}
  then there exists a label permutation $J$ such that $\hat\Sigma$ and $\Sigma J$ agree on $[n]\setminus (\bigcup_i S_i)$.
\end{lemma}

\begin{lemma}[Eigenvalue counts and eigenspaces with the empirical mean degree]\label{lem:empirical_degree}
\par\noindent
Under the assumptions of Proposition~\ref{prop:eigendecomposition_H}, set $\widehat\zeta_n=\hat d/\sqrt{\log n}$. With high probability, $H(\pm\sqrt{\hat d})$ has exactly $r_\pm$ eigenvalues below $-\widehat\zeta_n$. Let $V^\pm$ and $\widehat V^\pm$ have orthonormal columns spanning the subspaces associated with the $r_\pm$ smallest eigenvalues of $H(\pm\sqrt d)$ and $H(\pm\sqrt{\hat d})$, respectively. There is an orthogonal matrix $R^\pm\in\dR^{r_\pm\times r_\pm}$ such that
\[
 \|\widehat V^\pm-V^\pm R^\pm\|_F
  =O\left(\frac{\sqrt{r_\pm}\,\delta_{\mathrm{deg}}}{d (1-c)^2}\right),
 \qquad \delta_{\mathrm{deg}}=n^{-1/2}(\log n)^{O(1)}.
\]
\end{lemma}
\begin{proof}
Since $\hat d=(2/n)\sum_{x<y}A_{xy}$,
\[
 \mathbb E\hat d=d-\frac1n\sum_a\pi_aP_{aa},
 \qquad \operatorname{Var}(\hat d)\le\frac{2d}{n}.
\]
 Bernstein's inequality therefore gives
\[
 |\hat d-d|
 =O\left(\sqrt{\frac{d\log n}{n}}+\frac{\log n+p_{\max}}n\right)
\]
with high probability. Applying the same inequality to each vertex degree and taking a union bound gives $\|A\|\le\max_x\deg(x)=O(d+\log n)$. Since $\hat d/d\to1$, the following bound holds for both signs:
\begin{align}\label{eq:empirical_H_perturbation}
 \|H(\pm\sqrt{\hat d})-H(\pm\sqrt d)\|
 &\le|\hat d-d|+\frac{|\hat d-d|}{\sqrt{\hat d}+\sqrt d}\|A\|\notag\\
 &\le C|\hat d-d|\left(1+\frac{d+\log n}{\sqrt d}\right)
   =n^{-1/2}(\log n)^{O(1)}=:\delta_{\mathrm{deg}}.
\end{align}
By Proposition~\ref{prop:eigendecomposition_H} and \eqref{eq:signal_center_gap}, the informative eigenvalues of $H(\pm\sqrt d)$ are below $-d(1-c)^2/8$. For the remaining eigenvalues, Lemma~\ref{lem:lower_bound}(iv) with $\delta=C\sqrt d\,\log L_n\,\log\log n/\log n$ and the non-backtracking bound in Lemma~\ref{lem:SZ_precise_inputs}(ii) give the lower bound $-2Cd\,\log L_n\,\log\log n/\log n$. Since $\log L_n\le K\log\log n$ and $\hat d/d\to1$, we have
\begin{align*}
 2Cd\,\frac{\log L_n\,\log\log n}{\log n}+\delta_{\mathrm{deg}}
 &=o\left(\frac d{\sqrt{\log n}}\right), \quad
 \widehat\zeta_n=(1+o(1))\frac d{\sqrt{\log n}}=o(d).
\end{align*}
Weyl's inequality therefore shows that $H(\pm\sqrt{\hat d})$ has exactly $r_\pm$ eigenvalues below $-\widehat\zeta_n$.

For every unit eigenvector $v$ of $H(\pm\sqrt d)$ with informative eigenvalue $\nu$,
\[
 \|(H(\pm\sqrt{\hat d})-\nu I)v\|
 =\|(H(\pm\sqrt{\hat d})-H(\pm\sqrt d))v\|
 \le\delta_{\mathrm{deg}}.
\]
Weyl's inequality and the lower bound on the noninformative eigenvalues give
\[
 \lambda_{n-r_\pm}\bigl(H(\pm\sqrt{\hat d})\bigr)
 \ge-2Cd\,\frac{\log L_n\,\log\log n}{\log n}-\delta_{\mathrm{deg}}
 \ge-\frac{d(1-c)^2}{16}
\]
for all sufficiently large $n$. The last inequality holds because the two error terms sum to $o(d)$ and $c<1$ is fixed. Since $\nu\le-d(1-c)^2/8$, every noninformative eigenvalue lies above $\nu$, and
\[
 \min_{1\le j\le n-r_\pm}
 \left|\lambda_j\bigl(H(\pm\sqrt{\hat d})\bigr)-\nu\right|
 \ge\frac{d(1-c)^2}{8}-\frac{d(1-c)^2}{16}
 =\frac{d(1-c)^2}{16}.
\]
Now choose the columns $v_1^\pm,\ldots,v_{r_\pm}^\pm$ of $V^\pm$ to be eigenvectors of $H(\pm\sqrt d)$. Applying Lemma~\ref{lem:davis_kahan} to $H(\pm\sqrt{\hat d})$ and its invariant subspace $\operatorname{im}\widehat V^\pm$ gives
\[
 \bigl\|(I-\widehat V^\pm(\widehat V^\pm)^\top)v_k^\pm\bigr\|
 \le\frac{16\delta_{\mathrm{deg}}}{d(1-c)^2},
 \qquad 1\le k\le r_\pm.
\]
Summing the squares yields
\begin{align*}
 \bigl\|(I-\widehat V^\pm(\widehat V^\pm)^\top)V^\pm\bigr\|_F^2
 &=\sum_{k=1}^{r_\pm}
   \bigl\|(I-\widehat V^\pm(\widehat V^\pm)^\top)v_k^\pm\bigr\|^2\le\frac{256r_\pm\delta_{\mathrm{deg}}^2}{d^2(1-c)^4}.
\end{align*}
The last conclusion of Lemma~\ref{lem:davis_kahan} now gives an orthogonal matrix $R^\pm$ such that
\[
 \|\widehat V^\pm-V^\pm R^\pm\|_F
 \le\sqrt2\bigl\|(I-\widehat V^\pm(\widehat V^\pm)^\top)V^\pm\bigr\|_F
 \le\frac{16\sqrt{2r_\pm}\,\delta_{\mathrm{deg}}}{d(1-c)^2}.
\]
This bound holds for any orthonormal bases of the two subspaces, since changing either basis leaves the minimum over $R^\pm$ unchanged.
\end{proof}

\begin{proof}[Proof of Theorem~\ref{thm:weak_recovery}]
The assumption $\max_{i\le r}\tau_i\le c<1$ gives $r_0=r$. By Lemma~\ref{lem:empirical_degree}, Algorithm~\ref{alg:spectral_clust} selects exactly $r_+$ and $r_-$ columns. Let $\widehat V^\pm$ be the matrices of these eigenvectors, computed from $H(\pm\sqrt{\hat d})$, and let $V^\pm$ be the corresponding matrices computed from $H(\pm\sqrt d)$. The same lemma gives orthogonal matrices $R^\pm$ such that
\[
 \begin{aligned}
 &\bigl\|[\widehat V^+,\widehat V^-]-[V^+R^+,V^-R^-]\bigr\|_F=O\left(\frac{\sqrt r\,\delta_{\mathrm{deg}}}{d(1-c)^2}\right)
 =o\left(\sqrt{r/d}\right).
 \end{aligned}
\]
Write $V=[\widehat V^+,\widehat V^-]$ for the matrix used by Algorithm~\ref{alg:spectral_clust}. Combining this estimate with Lemma~\ref{lem:bethe_overlaps_sharp} gives an orthogonal matrix $O$ such that
\[
 \|V-\check\Phi O\|_F
 \le2\left(\sum_{i=1}^r\tau_i\right)^{1/2}
   +\frac{C\sqrt r(\sqrt d+1)}{d(1-c)^2}.
\]
Since $Q$ is nonnegative with row sums $d$, $|\mu_i|\le d$ and
\[
 \tau_i\ge\frac1d,\qquad
 \frac rd\le\sum_{i=1}^r\tau_i.
\]
Thus, for a constant $C_0=C_0(c)$,
\[
 \|V-\check\Phi O\|_F^2\le C_0\sum_{i=1}^r\tau_i.
\]

If $r=1$, every vertex receives the sole label and the overlap is one, so assume $r\ge2$. Write $\check\Phi=n^{-1/2}\Sigma\Phi$, where $\Sigma_{xk}=\mathbf1_{\{\sigma(x)=k\}}$ and the columns of $\Phi$ are the eigenvectors of $Q$. Lemma~2.1 in \citep{lei2015consistency} gives the admissible center separations $\delta_k=(n\pi_k)^{-1/2}$ in Lemma~\ref{lem:k-means-consistency}. In particular, $n\pi_k\delta_k^2=1$ for every $k$.

Choose $C=(16+8\eta)C_0$. If $C\sum_i\tau_i\ge1$, the claimed lower bound follows from $\operatorname{ov}(\sigma,\hat\sigma)\ge0$. Otherwise,
\[
 (16+8\eta)\|V-\check\Phi O\|_F^2
 \le C\sum_i\tau_i<1=n\pi_k\delta_k^2,
\]
so Lemma~\ref{lem:k-means-consistency} supplies a label permutation and sets $S_k$ outside which the labels agree. Consequently,
\begin{align*}
 1-\operatorname{ov}(\sigma,\hat\sigma)
 &\le\frac1n\sum_k|S_k|
 \le\sum_k\frac{|S_k|}{n\pi_k}\le(16+8\eta)\|V-\check\Phi O\|_F^2
 \le C\sum_i\tau_i.
\end{align*}
\end{proof}

\section{Spectral stability of Hermitian matrices}\label{sec:app:perturbation}

We use the following block comparison and eigenvector estimate.

\begin{lemma}\label{lem:app:local_weyl}
Let $M\in\dR^{n\times n}$ be symmetric, let $1\le k\le n$, let $V\in\dR^{n\times k}$ satisfy $V^\top V=I_k$, and let $\Lambda=\operatorname{diag}(\lambda_1,\ldots,\lambda_k)$ with $\lambda_i\in\dR$. Set
\[
 P=VV^\top,\qquad
 S=V\Lambda V^\top+(I-P)M(I-P).
\]
Then
\[
 \|M-S\|\le2\|MV-V\Lambda\|.
\]
Consequently, there are $k$ eigenvalues $\nu_1,\ldots,\nu_k$ of $M$, counted with multiplicity, such that
\[
 |\nu_i-\lambda_i|\le2\|MV-V\Lambda\|,\qquad 1\le i\le k.
\]
\end{lemma}
\begin{proof}
The identity
\[
 M-S=(MV-V\Lambda)V^\top
       +V(MV-V\Lambda)^\top(I-P)
\]
gives the norm bound, since $\|V\|=\|V^\top\|=1$ and $\|I-P\|\le1$. Moreover, $SV=V\Lambda$, so $\lambda_1,\ldots,\lambda_k$ are eigenvalues of $S$ with the stated multiplicities. Weyl's inequality gives the conclusion.
\end{proof}

\begin{lemma}[A local Davis-Kahan theorem]\label{lem:davis_kahan}
Let $M\in\dR^{n\times n}$ be symmetric and let $E$ be an $M$-invariant subspace. Suppose $\lambda\in\dR$, $\|v\|=1$, and $\|(M-\lambda I)v\|\le\varepsilon$. If
\[
 \delta=\operatorname{dist}\!\left(\lambda,
             \operatorname{spec}(M|_{E^\perp})\right)>0,
\]
then $\operatorname{dist}(v,E)\le\varepsilon/\delta$.
For any $V,\widehat V\in\dR^{n\times k}$ with orthonormal columns, one also has
\[
 \min_{R^\top R=I_k}\|\widehat V-VR\|_F
 \le\sqrt2\,\|(I-\widehat V\widehat V^\top)V\|_F,
\]
where the minimum is over $R\in\dR^{k\times k}$.
\end{lemma}
\begin{proof}
The first assertion is immediate if $E=\dR^n$. Otherwise, since $M$ is symmetric and $E$ is invariant, its orthogonal projection commutes with $M$, and
\[
 \delta\|P_{E^\perp}v\|
 \le\|(M-\lambda I)P_{E^\perp}v\|
 =\|P_{E^\perp}(M-\lambda I)v\|
 \le\varepsilon,
\]
which proves the first assertion.
For the second, let $s_1,\ldots,s_k\in[0,1]$ be the singular values of $V^\top\widehat V$. Its singular value decomposition gives
\begin{align*}
 \min_{R^\top R=I_k}\|\widehat V-VR\|_F^2
 &=2\sum_{i=1}^k(1-s_i)\le2\sum_{i=1}^k(1-s_i^2)
 =2\|(I-\widehat V\widehat V^\top)V\|_F^2.
\end{align*}
Taking square roots proves the claim.
\end{proof}

\section{Proof of Corollary~\ref{cor:hatd}}\label{sec:app:cor}
\begin{proof}
Bernstein's inequality gives $|\hat d-d|=O(\sqrt{\log n/n})$ with high probability. Since $\|A\|\le\max_x\deg(x)=O(\log n)$, we obtain
\[
 \delta_{\mathrm{deg}}:=\max_{s=\pm1}\|H(s\sqrt{\hat d})-H(s\sqrt d)\|
 =O\left(\frac{\log^{3/2}n}{\sqrt n}\right)=o(\varepsilon_n).
\]
Since the parameters are fixed, Lemma~\ref{lem:SZ_precise_inputs}(ii) bounds the modulus of every eigenvalue of $\tilde B$, other than its informative outliers, by $\sqrt d+C\log\log n/\log n$. Parts (i), (ii), and (iv) of Lemma~\ref{lem:lower_bound} therefore give
\[
 \lambda_{n-r_\pm}\bigl(H(\pm\sqrt d)\bigr)
 \ge-C\frac{\log\log n}{\log n}.
\]
Since $\varepsilon_n=\log^{-a}n$ with $0<a<1$, Weyl's inequality yields, for all sufficiently large $n$,
\[
 \lambda_{n-r_\pm}\bigl(H(\pm\sqrt{\hat d})\bigr)
 \ge-C\frac{\log\log n}{\log n}-\delta_{\mathrm{deg}}
 >-\varepsilon_n.
\]
Thus at most $r_\pm$ eigenvalues lie below $-\varepsilon_n$. For the reverse bound, \eqref{eq:uniform_trial_negative} and the perturbation bound above give
\[
 y^\top H(\pm\sqrt{\hat d})y
 \le(-\gamma+\delta_{\mathrm{deg}})\|y\|^2
 <-\varepsilon_n\|y\|^2,\qquad 0\ne y\in E^\pm.
\]
The Courant--Fischer principle and $\dim E^\pm=r_\pm$ therefore give at least $r_\pm$ such eigenvalues. Hence
\[
 N_{\varepsilon_n}(\sqrt{\hat d})=r_+,
 \qquad N_{\varepsilon_n}(-\sqrt{\hat d})=r_-.
\]

It remains to count the negative eigenvalues of $H(\sqrt{\hat d}+\varepsilon_n)$ and $H(-\sqrt{\hat d}-\varepsilon_n)$. The Bernstein bound gives $|\sqrt{\hat d}-\sqrt d|=o(\varepsilon_n)$, so
 $\sqrt{\hat d}+\varepsilon_n-\sqrt d
 =\varepsilon_n+o(\varepsilon_n)$.
The noninformative eigenvalues of $\tilde B$ have modulus at most $\sqrt d+O(\log\log n/\log n)$, and $\log\log n/\log n=o(\varepsilon_n)$. Hence, for all sufficiently large $n$, every noninformative eigenvalue has modulus strictly smaller than $\sqrt{\hat d}+\varepsilon_n$. By Lemma~\ref{lem:lower_bound}(i) and (ii), respectively,
\[
 N(\sqrt{\hat d}+\varepsilon_n)\le r_+,
 \qquad N(-\sqrt{\hat d}-\varepsilon_n)\le r_-.
\]
For the reverse inequalities, note that
\[
 \left|\sqrt{\hat d}+\varepsilon_n-\sqrt d\right|
 \le |\sqrt{\hat d}-\sqrt d|+\varepsilon_n\longrightarrow0.
\]
Thus we may apply \eqref{eq:uniform_trial_negative} with $t=\sqrt{\hat d}+\varepsilon_n$. It shows that $H(\sqrt{\hat d}+\varepsilon_n)$ is negative definite on $E^+$, of dimension $r_+$, and that $H(-\sqrt{\hat d}-\varepsilon_n)$ is negative definite on $E^-$, of dimension $r_-$. The Courant--Fischer principle therefore gives
\[
 N(\sqrt{\hat d}+\varepsilon_n)\ge r_+,
 \qquad N(-\sqrt{\hat d}-\varepsilon_n)\ge r_-.
\]
Combining the upper and lower bounds proves the two remaining identities in Corollary~\ref{cor:hatd}.
\end{proof}

\end{document}